\documentclass[reqno, 12pt]{amsart}
\usepackage{amsmath, amssymb}
\usepackage[all]{xypic}
\newcommand{\eeq}{\end{equation}}

\usepackage{times}
\usepackage[T1]{fontenc}
\usepackage{mathrsfs}
\usepackage{epsfig}
\usepackage{color}

\newcommand{\bZ} {\mathbb{Z}_p}

\newcommand{\bQ}{\mathbb{Q}_p}

\newcommand{\bs}{{\bf s}}
\newcommand{\bx}{{\bf x}}

\newcommand{\sRS}{{\mathcal{R}(S)}}

\newcommand{\sC}{{\mathcal C}}

\newcommand{\CC}{\mathbb{C}}
\newcommand{\NN}{\mathbb{N}}
\newcommand{\QQ}{\mathbb{Q}}
\newcommand{\ZZ}{\mathbb{Z}}

\newcommand{\ev}{{\rm ev}}
\newcommand{\Et} {\tau}
\newcommand{\Es}{ \sigma}

\newcommand{\tauu}{{\rho}}
\newcommand{\taux}{{\rho}}
\newcommand{\muu}{{\nu}}
\newcommand{\sS}{{\mathcal S}}
\newcommand{\sM}{{ U}}
\newcommand{\sY}{{\mathcal N}}

\newtheorem{thm}{Theorem}[section]

\newtheorem{exam}[thm]{Example}

\newtheorem{rem}[thm]{Remark}
\newtheorem{defi}[thm]{Definition}

\newtheorem{prop}[thm]{Proposition}
\newtheorem{lem}[thm]{Lemma}

\def\slfrac#1#2{\hbox{\kern.1em %
 \raise.5ex\hbox{\the\scriptfont0 #1}\kern-.11em %
 /\kern-.15em\lower.25ex\hbox{\the\scriptfont0 #2}}}

\title{A Skolem-Mahler-Lech theorem for iterated automorphisms of $K$-algebras}

\author{Jason P. Bell}
\thanks{The research of the first-named author  was supported by NSERC grant 611456.}
\address{ Department of Mathematics, University of Waterloo, 
Waterloo, ON, N2L  3G1, CANADA}
\email{jpbell@uwaterloo.ca}

\author{Jeffrey C. Lagarias}
\thanks{The research of the second author was supported by NSF grants 
DMS-0801029 and DMS-1101373.}

\address{Department of Mathematics, University of Michigan, 530 Church Street,
Ann Arbor, MI 48109-1043, USA}
\email{lagarias@umich.edu}

\date{November 21, 2013, rev6}

\begin{document}

\begin{abstract} 
This paper proves  a commutative algebraic extension 
 of a generalized Skolem-Mahler-Lech theorem due to the first
author.
Let  $A$ be a finitely generated commutative $K$-algebra
over a field of characteristic $0$, and let $\Es$ be  
a $K$-algebra automorphism  of $A$.
Given ideals $I$ and $J$ of $A$, we show that  
 the set $S$ of integers $m$ such that 
$ \Es^m(I) \supseteq J$ is a finite union of 
complete doubly infinite arithmetic progressions in $m$, up to the addition of a finite set. 
Alternatively, this  result  states that for an affine scheme $X$ of finite type over $K$,
an automorphism $\sigma \in {\rm Aut}_K(X)$, and $Y$ and $Z$ any two closed subschemes of $X$,  the set 
of integers $m$ with $\sigma^m(Z ) \subseteq Y$ is as above.
The paper presents  examples
showing that this result may fail to hold if the affine scheme $X$ is 
not of finite type, or if $X$ is of finite type but  the field $K$ has positive characteristic.
\end{abstract}


\maketitle
\tableofcontents

%
%
%
%
\section{Introduction}

The Skolem-Mahler-Lech theorem is a fundamental result,
which  characterizes
the structure of the  set of zeros of a linear recurrence;
we term the resulting  structure  the SML property below.
The paper is motivated by  the question:  ``What is the maximal level of generality
of the Skolem-Mahler-Lech theorem, in which its conclusion holds?"    
We view  this question in the general framework of orbits of dynamical systems
of algebraic type.

In 2006 the  first author gave an algebro-geometric generalization of the
Skolem-Mahler-Lech theorem which applied to orbits of an automorphism
of an affine variety acting on  geometric points. 
The object of this paper is to show there is a
further algebro-geometric generalization of the
Skolem-Mahler-Lech theorem that applies  to automorphisms
on the coordinate ring of the variety acting  at  the level of ideals.
The new result may be interpreted as a result about automorphisms
of affine schemes.
To state it, we first review the successive generalizations of the Skolem-Mahler-Lech
theorem. 
%
%
%
%
\subsection{The Skolem-Mahler-Lech theorem}\label{sec10}

The Skolem-Mahler-Lech theorem (hereafter abbreviated to SML theorem) is
a fundamental result which can be stated in several apparently
different forms.  The  original formulation  of the SML theorem  
concerned the zeros of power series coefficients of rational functions.

\begin{thm} {\em (SML theorem-rational function form)\,} \label{thm:SML2}
 Let  $K$ be a field of characteristic zero, and let  $G(x) \in K(x)$ be a rational function
 that is finite at $x=0$.  If the  Taylor expansion
$G(x) \in K[[x]]$ at $x=0$  is given as
$$
G(x) = \sum_{k= 0}^{\infty} a_k  x^k,
 $$
 then the set of its zero coefficients
 $$
 S :=\{ k \in \NN: a_k =0\}
 $$ 
  is a union of a finite number
of arithmetic 
progressions $\{ an+b: n \ge 0\}$,  with  $a, b>0$,
together with a (possibly empty) finite set. 
\end{thm}

This result was established  by Skolem \cite{Sk34} in 1934 for the case of
rational coefficients, using a $p$-adic  method he introduced in  \cite{Sk33}.
It was extended by Mahler \cite[Sect. 6]{Mah35} in 1935 to series
to coefficients in algebraic number fields,  and much later in 1956 (\cite{Mah56})  to coefficients
in the complex field  $\CC$. 
 In this last extension Mahler was  unaware 
 of  the 1953 work of Lech \cite{Le53} discussed below, see \cite{Mah57}.

There are now many different proofs and extensions of 
the Skolem-Mahler-Lech theorem in the
literature \cite{Bez89, Han86, vdP89, vdPT75}, some described in 
the  book of  Everest et al. \cite[Chap. 2.1]{EVSW03}.  The result
is valid in any field of characteristic zero, however  all known 
proofs use $p$-adic methods.

The  property of the set $S$ formulated in the conclusion  of 
all versions of the Skolem-Mahler-Lech theorem
can be stated as follows.


\begin{defi}\label{de01}
{\em 
 A  set  of natural numbers $S \subset \NN$ has the {\em (one-sided)   SML property} 
if it is a finite union of one-sided arithmetic progressions $\{ an+ b: n \ge 0\}$ with $a, b >0$,
augmented by a finite set, possibly empty. (We allow the possibility $b > a$ in the 
one-sided arithmetic progressions, i.e. some initial terms of the complete progression in $\NN$
may be omitted.)
}
\end{defi}

A second form of the Skolem-Mahler-Lech theorem concerns zeros of recurrence sequences.
This version  was first formulated by Lech \cite{Le53} in 1953. 

\begin{thm} {\em (SML theorem-recurrence  form)\,} \label{thm:SML}
 Let $K$ be a field of characteristic zero and let $f:\mathbb{N}\to K$ be a $K$-valued sequence
 that satisfies  a linear recurrence
 $$
 f(n)= a_1 f(n-1) + a_2 f(n-2) + \cdots + a_r f(n-r),  ~~~~   a_i \in K.
  $$
 for all $n \ge r$. Then the set  of zeros of the recurrence, 
 $$
 S  := \{ n \ge 1:  ~~f(n) =0\}.
 $$ 
 has the SML property.
\end{thm}

We can view the recurrence version of the SML theorem as asserting a property of a forward orbit
of a discrete dynamical system given by the linear recurrence. 
Note that the  set $S$ can also be interpreted as
the intersection of the forward orbits of two different dynamical systems, the first being the recurrence
$f(n)$ and the second  being
the constant linear recurrence $g(n)=0.$

The recurrence form of the SML theorem can be recast 
in a third equivalent form, which concerns iteration of linear maps,
 and encodes  containment relations of orbits of 
invertible linear maps under iteration in a  subspace
(\cite[Theorem 1.2 and Sect. 2]{Be06}).


\begin{thm} \label{th14} 
{\em (SML theorem-linear map form)}
Let $K$ be a field of characteristic $0$, and let
 $\sigma: K^n\rightarrow K^n$ be an invertible
linear map.
If  ${\bf v}\in K^n$, and  $W$ is  a vector subspace of
$K^n$ of codimension $1$, then  the set 
$$
S := \{ n \in {\mathbb N}:  \sigma^n({\bf v})\in W\}
$$
has the SML property. 
\end{thm}

In 2006 the first author \cite{Be06, Be06b}
found a  generalization of the SML theorem
to algebraic dynamics, which
 applies at the geometric level 
 to iteration of  automorphisms on affine algebraic
varieties.  It  concerns two-sided infinite sequences of iterates, and to state it  we extend the
notion of SML property to this case. 

%
%

\begin{defi}\label{de02}
{\em  A set of integers $S \subset \ZZ$ is said to have the {\em two-sided SML property}
if it is a finite union of complete arithmetic $\{ an+b: n \in \ZZ\}$
progressions on $\ZZ$ augmented by a finite set.  (That is,  we allow arithmetic
progression with $a=0$ which give one element sets.)
}
\end{defi}

The  generalization of the SML theorem to 
automorphisms on affine varieties is as follows.

%
%
\begin{thm} \label{thm:1}
{\em (Generalized SML Theorem for Affine Varieties)} 
Let   $K$ be a 
 field of characteristic zero, and 
 let $X$ be an affine $K$-variety and let $\sigma$ be an automorphism of $X$.  
 If $x$ is a $K$-point of $X$ and $Y$ is a subvariety of $X$ (Zariski closed subset of $X$), 
 then the set  
 $$
 S := \{ n \in \ZZ:  \sigma^n(x) \in Y\}
 $$
 has the two-sided SML property.
\end{thm}
\noindent

Despite its restriction to two-sided infinite sequences,
 Theorem \ref{thm:1} recovers the original (one-sided) 
recurrence form of  the Skolem-Mahler-Lech theorem above, via 
its implication of  the 
linear map form (Theorem \ref{th14}).
The proof of Theorem \ref{thm:1} in \cite{Be06, Be06b} again relies  on a  $p$-adic result, 
the {\em $p$-adic analytic arc theorem}, first established in \cite{Be06, Be06b} and then strengthened in \cite{BGT}. 
This latter result was recently further  strengthened  by Poonen \cite{Poo13}.
 
We also note that since the conclusion of Theorem \ref{thm:1}   is purely set-theoretic
and is preserved under extension of scalars, the result for a general  field $K$
follows  easily from the
special  that $K$ is an  algebraically closed field of characteristic zero.

Subsequent work of the first author with Ghioca and Tucker  \cite[Theorem 1.3]{BGT},
established an analogous theorem  for  forward orbits of 
\'{e}tale endomorphisms  of quasiprojective varieties  defined over $\CC$.  We 
also point out that Denis \cite{Den94} had earlier proved 
a special case of this result for \'etale self maps of $\mathbb{P}^n_{\mathbb{C}}$.   
Recently 
Sierra \cite[Conjecture 5.15]{Si11} suggests a possible extension of these results that 
would have interesting consequences for the algebras studied in noncommutative projective geometry.

%
%
%
%
\subsection{Main result}\label{sec11}

Our starting point  is the observation that Theorem \ref{thm:1} can be recast in purely algebraic terms. 
To do this,  suppose  that  $K$ is an algebraically closed field of characteristic zero.  
In classical affine algebraic geometry, we then have a  contravariant equivalence of categories 
$$\left\{{\rm Affine ~}K{\textrm{-Varieties}}\right\} \longleftrightarrow \left\{ {\rm Reduced~finitely~generated~}K\textrm{-algebras}\right\}
$$ 
induced by the functor which takes an affine variety $X$ to its coordinate ring $K[X]$ and which takes a morphism $\phi : X\to Y$ of affine varieties to the homomorphism $\phi^*:K[Y]\to K[X]$ of $K$-algebras given by $\phi^*(f)=f\circ \phi$ ( \cite[Corollaries 1.4 and 3.8]{Har77}).  The Nullstellensatz gives further information, 
showing that there is an inclusion reversing bijection between the subvarieties of an affine variety $X$ and the radical ideals of its coordinate ring $K[X]$.  
In particular, for a 
 morphism $\phi: X \to X$  the inclusion
$\phi (V(I)) \subseteq V(J)$ of zero sets of  radical ideals  $I$, $J$ corresponds to 
the reversed algebraic inclusion  $\phi^{\star}(I) \supseteq J$.
It follows that Theorem \ref{thm:1} restricted to algebraically closed fields $K$ may  be recast  in algebraic terms as: 
 {\em If  $A$ is a finitely generated reduced commutative $K$-algebra with a $K$-algebra automorphism $\sigma$ 
and $M$ and $J$ are ideals of $A$ with $M$ a maximal ideal and $J$ a radical ideal, then the set of integers $n$ for which 
$\sigma^{n}(M)\supseteq J$ has the two-sided SML property.}

From this algebraic perspective, it is natural to ask whether a similar result holds for more general ideal inclusions 
in finitely generated commutative $K$-algebras.  
Our main result  gives an affirmative answer to this question.

%
%

\begin{thm} \label{thm:main}
{\em (Generalized SML theorem for ideal inclusions)}
Let $K$ be any field of characteristic zero  and let $A$ be a finitely generated commutative
$K$-algebra.  If $\sigma: A\to A$ is a $K$-algebra automorphism and $I$ and $J$ are ideals of $A$, then 
$$
S = S(I, J) := \{n\in \mathbb{Z}~:~  \sigma^n(I) \supset J\}
$$
has the two-sided SML property. Here
$\sigma(I) := \{ \sigma(a): a \in I\}$.  
 \end{thm}

The two-sided SML property  of  this result also holds for
 the set of integers $n$ for which  $\sigma^n(J)\subseteq I$,
because  this condition holds if and only if $\sigma^{-n}(I)\supseteq J$, 
where $ \sigma^{-1}(I) = \{ b: \sigma(b) \in I\}$; apply Theorem \ref{thm:main}
 to the automorphism $\sigma^{-1}$.

Theorem \ref{thm:main}  implies Theorem \ref{thm:1} and hence
 all the earlier forms of the SML theorem stated above.
In  the special case that   $I$ is a maximal ideal,
 Theorem \ref{thm:main}  is itself deducible from the algebraic form of Theorem \ref{thm:1}.  
The  main content of Theorem \ref{thm:main}, and also the source of the difficulty  in proving it, revolves around relaxing the maximality condition on $I$.
 For example, if $A=\mathbb{C}[x,y]$, then a polynomial $p(x,y)\in A$ is in the maximal ideal $(x,y)$ if and only if $p(0,0)=0$. 
  On the other hand, $p(x,y)$ is in the ideal $ (x,y^2)$ if and only if $p(0,0)=0$ and $\partial{p(x,y)}/\partial{y}$ vanishes at $(0,0)$. 
    In this sense,  the ideal $I=(x,y^2)$ encodes additional ``infinitesimal'' information.  
    It seems easier to understand the behavior of the iterates of an element $p(x,y)$ of $A$ under an automorphism $\sigma$ with 
    respect to evaluation at a given point than to understand how the iterates behave with respect to more complicated conditions 
    involving the vanishing of partial derivatives at a point.  This difficulty and, more generally, 
    the possible occurrence of nil ideals and non-radical ideals in $A$ adds an extra level of complication to the problem.
     (see Examples \ref{exam2} and \ref{exam3}).
    
    In Section \ref{sec6} we 
 examples showing that this result is strictly more general than Theorem \ref{thm:1}.
In particular, there exist two ideals $I_1, I_2$ having the same radical ideal $I = \sqrt{I_1}= \sqrt{I_2}$
and $J$ such the sets $S(I_1, J) \ne S(I_2, J)$, see \ref{exam1}.
  
  As with previous work, our method  to prove Theorem \ref{thm:main} makes use of  a $p$-adic result.
 We present a $p$-adic interpolation result that  strengthens the  
  $p$-adic analytic arc theorem established in  \cite{Be06b}, which we call the 
 {\em  generalized $p$-adic analytic arc theorem}  (Theorem \ref{thAA}).
This result concerns polynomial automorphisms  $\sigma: \ZZ_p[x_1, \ldots, x_d] \to \ZZ_p[x_1, \ldots, x_d]$ 
and applies  for $p \ge 5$.  It asserts that for 
a   $\ZZ_p$-algebra $S$ that is finitely generated and is a torsion-free $\ZZ_p$-module, the induced (nonlinear) map 
$f_{\sigma} :S^d\to S^d$ has the property that 
 iterates of a suitable initial point ${\bf s}=(s_1,\ldots ,s_d)\in S^d$ under $\sigma$
 can be embedded in a $p$-adic analytic arc.  
 Note that $S$ is a free $\ZZ_p$-module of finite rank $r>0$, and one might 
 think that this extension can be obtained by fixing a 
 $\ZZ_p$-module isomorphism between $S$ and $(\ZZ_p)^r$ and then applying the 
  analytic arc theorem of \cite{Be06} separately to each coordinate.  
  This is not the case, however, as application of the map $\sigma$
 involves expressions that include combinations of elements from different coordinates.
 Our innovation to  get around this obstacle consists of  working with a larger ring of functions than usual; 
 namely, the subset $\sRS$ of $S \otimes_{\ZZ_p} \QQ_p[z] $ that
 consists of all polynomials which map $\ZZ_p$ into $S$, and 
 obtaining analytic maps via successive approximations via functions in this ring.
 The commutative ring $\sRS$ is rather pathological, it is not Noetherian and is not of
 finite type over $\ZZ_p$.
 Our  proof requires  establishing a nice  property of  certain subalgebras of $\sRS$,  given in Lemma~\ref{le34}.
 As in the case of the $p$-adic analytic arc theorem  treated in
 \cite{Be06},  the generalized $p$-adic analytic arc theorem  fails to hold for $p=2$, 
  and its truth for $p=3$ remains open.

  The proof of  Theorem \ref{thm:main}  is  
  commutative algebraic, aside from the 
  $p$-adic interpolation result.
    A key idea is to first establish a result  for the ring $A= \ZZ_p[x_1, \ldots , x_d]$,
  treating  only for  the special  case that $I ,J$ are  $p$-reduced ideals  (Definition \ref{def42}) when $p \ge 5$ (Theorem \ref{th41}).
  The $p$-reduced condition 
  is needed in order to apply  the generalized $p$-adic analytic arc
  theorem. Here we also use an idea of Amitsur (Lemma \ref{lem: amitsur}), to handle the case of $p$-reduced ideals $I$
   where $(A/I )\otimes \QQ_p$ is an infinite-dimensional $\QQ_p$-vector space.
We next deduce  the result for arbitrary ideals $I, J$ in a polynomial ring $A=K[x_1, \ldots , x_d]$ over
  a field  of characteristic $0$
  (Theorem \ref{thm: polycase}), 
  as follows.  We  first show that 
  $K$ can be taken to be finitely generated over $\QQ$; next,  we use the fact that
  such a field embeds
    into infinitely many $p$-adic fields $\QQ_p$ via the Chebotarev density theorem; finally, 
  passing from $\QQ_p$ to $\ZZ_p$-coefficients, we show  that one can pass  from ideals to suitable
  $p$-reduced ideals.    
  We then treat the general case that $A$ is a finitely-generated commutative $K$-algebra, by  using a theorem of 
   Srinivas \cite[Theorem 2]{Sr91} to reduce this case  to that  of a polynomial ring $K[x_1, \ldots, x_d]$.
   
   An interesting feature of this proof is that  while Theorem \ref{thm:main} 
   concerns finitely-generated commutative $K$-algebras $A$
   which are Noetherian rings,  the proof  itself currently requires a detour using  the non-Noetherian ring $\sRS$.
   %
 
\subsection{Affine scheme version of main result}\label{sec12new}
We can formulate the main result  in the category of affine schemes. 
 There is a standard correspondence of categories
$$
\left\{{\rm Affine ~ Schemes}\right\} \longleftrightarrow \left\{ {\rm Commutative~rings~with ~identity}\right\}.
$$
and  using it Theorem \ref{thm:main} may be restated as follows. 
%
%

 \begin{thm} \label{th-scheme}
 {\em Generalized SML theorem for  affine schemes of finite type)}
Let $K$ be a field of characteristic zero and let $X$ be an affine scheme of finite type over $K$.
   If $\Es \in {\rm Aut}_K(X)$ and $Y$ and $Z$ are closed subschemes of $X$, then 
$$S= S(Z,Y)  := \{n \in \mathbb{Z} ~:~ \Es^{n}(Z)\subseteq Y\}$$ 
has the two-sided SML property.
\end{thm}

From the scheme-theoretic viewpoint
Theorem \ref{thm:1} corresponds to the special case that $Y$ and $Z$ are 
reduced 
closed subschemes and $Z=\{ x\}$ is a point. 
The main difficulty  in establishing Theorem \ref{th-scheme} is to omit the requirement that
the source  $Z$ and target $Y$ be reduced,
while allowing a source $Z$ of arbitrary dimension is less of a difficulty.

We show in Section \ref{sec6} via examples  that for automorphisms one cannot  
generalize this result to arbitrary affine schemes without
altering the conclusion that  the set is a  two-sided SML set.
Example \ref{exam4} 
shows the need  for the scheme to be of finite type over $K$, while 
example \ref{exam5} shows the need to work over a field $K$ of characteristic zero.
There remains  a possibility of generalizing the result to arbitrary schemes of finite type 
over characteristic
zero fields.

%
%

\subsection{Generalizations}\label{sec12a}

The ultimate level of generalization possible for  mappings having the SML property
is not clear. 

A natural question is:  ``Is there an SML theorem for endomorphisms
of algebraic varieties?"  
For endomorphisms  that are not automorphisms, the maps $\Et^m$ with $m<0$ are not defined, and 
extensions of Theorem ~\ref{thm:main}  to endomorphisms must be formulated in terms of
 one-sided infinite arithmetic progressions.  
  Recent work described below shows that the SML property does hold  for some 
  classes of endomorphisms at the geometric level, while endomorphisms at the algebraic level
  have not yet been studied.  At present in  the characteristic $0$ case no counterexamples to the 
  (one-sided) SML property are known
 for any endomorphism at either   the geometric and algebraic level.

The  dynamics of endomorphisms in the geometric setting
is currently  a very  active area of study.
 In 2009 Ghioca and Tucker \cite{GT09} 
conjectured that 
the forward iterates of a point $P$ under an endomorphism
 $\tau: X \to X$ of a quasiprojective variety over $\CC$ should satisfy
the SMLproperty for intersecting a closed subvariety $V$.
  They term this assertion the
{\em dynamical Mordell-Lang conjecture}. This conjecture  fits in the general framework of dynamical conjectures of
S. Zhang \cite[Sect. 4]{Zha06}.  
It  is  now known that one-sided versions
of the Skolem-Mahler-Lech theorem are valid for some special classes of  endomorphisms---see for example
 Benedetto, Ghioca,  Kurlberg and Tucker \cite{BGKT12}.
 The  dynamical Mordell-Lang conjecture  currently appears to be
difficult in the general case.

 When endomorphisms are viewed at the ideal level, as in this paper, 
 there are two different questions to consider:
 the first concerns upward inclusions
and the second concerns  downward inclusions of
ideals under preimages of 
endomorphisms of polynomial rings. 
Given an  endomorphism $\Et: A \to A$ of a commutative $K$-algebra $A$, 
 and  a nonzero ideal $I$ of $A$ the preimages $\Et^{-1}(I)$
 are ideals of $A$,
but the image $\Et(I)$ is 
usually not an ideal of $A$ but is  an ideal of the image ring $A':= \Et(A)$.
To formulate   inclusions purely in terms of $A$-ideals, these  concern the sets 
\[
S(I,J):= \{m \ge 0~|~I \supseteq (\Et^m)^{-1}(J) \}
\]
and
\[
S'(I,J) := \{m \ge 0~|~(\Et^m)^{-1}(I) \supseteq J\},
\]
respectively.

Our methods in this paper do not give information about  either of the sets $S(I,J)$ and $S'(I, J)$.
However a  number of results in this paper partially  extend to the endomorphism case, see
Remarks \ref{rem22} and \ref{rem36b}. 
New ideas are certainly needed for general endomorphisms, 
because the  generalized analytic arc theorem
is known not to hold  in the $p$-adic case near  some superattracting fixed points.\\

\paragraph{\bf Acknowledgments.} We thank Harm Derksen
for comments on the nonzero characteristic case and Brian Conrad
for comments on the scheme theory viewpoint. We thank  
 the reviewer for  helpful comments.

%
%
%

\section{$p$-Adic Preliminaries}\label{sec2}

Much of our study of automorphisms of algebras relies on reductions to simpler cases. 
We will eventually 
 reduce the proof of Theorem \ref{thm:main} to the special case that one is working with a 
 polynomial ring over a $p$-adic field.
 
 For the analysis of this special case  we establish
  preliminary results on Jacobian matrices of $p$-adic mappings and some results from $p$-adic analysis.
  Section \ref{sec21} allows $S$-algebra endomorphisms, while all  later sections specialize
  to the automorphism case. 
    
 %
%
%
%
\subsection{Evaluation maps and Jacobians for polynomial endomorphicms}\label{sec21}

Let 
$R=S[x_1, \ldots, x_d]$ be a polynomial ring over an integral
domain $S$. Any  $S$-algebra endomorphism $\Et: R \to R$
 is uniquely
determined by its values $\{\Et(x_i): 1 \le i \le d\}$,
on the monomials $x_i$, and any assignment of  values 
\[
\Et(x_i) := F_i(x_1,\ldots, x_d) \in S[x_1, \ldots, x_d], \quad 1 \le i \le d,
\]
uniquely extends to an  endomorphism $\Et: R \to R$
that acts as the identity on $S$,  given by
\begin{equation}\label{201}
\Et(P(x_1, \ldots , x_d)) := P(\Et(x_1),\ldots, \Et(x_d)) \in S[x_1, \ldots, x_d].
\end{equation}
Composition of endomorphisms will be denoted $\tau_2 \circ \tau_1(P) := \tau_2(\tau_1(P)).$
In what follows  we use   $\Et$ to denote a  general $S$-algebra endomorphism, 
while symbols $\Es, \tauu$ are reserved for $S$-algebra automorphisms.

 The {\em Jacobian matrix} 
$J(\Et; \bx) \in M_{d \times d}(S[x_1,\ldots,x_d])$ of the map $\Et$ at $\bx$ is given by
\[
J(\Et; \bx):=
\left[ \begin{array}{cccc} 
\frac{\partial}{\partial x_1} F_1 &  \frac{\partial}{\partial x_2} F_1
&  \cdots & \frac{\partial}{\partial x_d} F_1 \\
\frac{\partial}{\partial x_1} F_2 & \frac{\partial}{\partial x_2} F_2 
& \cdots & 
\frac{\partial}{\partial x_d} F_2 \\
& \cdots & \cdots &  \\
\frac{\partial}{\partial x_1} F_d & \frac{\partial}{\partial x_2} F_2 
& \cdots & 
\frac{\partial}{\partial x_d} F_d\\
\end{array} \right].
\]
These matrices satisfy the polynomial identity under composition
\[
J(\Et_2 \circ \Et_1; \bx)= J(\Et_2; \Et_1(\bx)) J(\Et_1; \bx),
\]
in the ring $M_{d \times d}(S[x_1,\ldots,x_d])$.

For $\bs=(s_1, s_2, \ldots, s_d) \in S^d$,
we define the {\em evaluation map} $\ev_{\bs}: S[x_1, x_2, \ldots, x_d] \to S$, which  assigns
$x_i \mapsto s_i$, i.e., 
\[
\ev_{\bs}(F)(x_1, \ldots,x_d) := F(s_1, s_2, \ldots,s_d).
\]
(For constants $c \in S$ we have $\ev_{\bs}(c)(x_1, .., x_d) = c.$)
Using this map, we define 
the Jacobian matrix of a 
polynomial map  $\Et$ with its entries evaluated at the point 
${\bf s}\in S^d$ as:
$$
J(\Et; {\bf s}) := \ev_{\bf s}(J(\Et; \bx)) = \left[ \ev_{\bf s} \Big(\frac{\partial}{\partial x_{j}}F_i(x_1, \ldots , x_d)\Big)\right]_{1 \le i, j \le d}, 
$$ 
with   $J(\Et; {\bf s}) \in M_{d \times d}(S)$.

In the same vein, given an endomorphism $\Et$ of $S[x_1, \ldots ,x_d]$ 
we obtain a 
map 
$f_{\Et}: S^d \to S^d$ defined for each $\bs \in S^d$, acting coordinatewise,  by
\[
f_{\Et}(\bs)
:= \ev_{\bs}(\Et(x_1), \Et(x_2), \cdots, \Et(x_d)) 
=(F_1(s_1,\ldots,s_d), F_2(s_1,\ldots, s_d), \ldots, F_d(s_1, \ldots,s_d)).
\]
We call $f_{\Et}$ the {\em dynamical evaluation map} 
associated to  $\Et$.  
The important property it has   is compatibility  with  composition of maps, given as:
\begin{equation}~\label{eq429}
f_{\Et_2 \circ \Et_1}(\bs) = 
(f_{\Et_2} \circ f_{\Et_1})(\bs).
\end{equation}
In particular, for  iteration of the map $f_{\Et}$
one has 
$f_{\Et^{m}}= (f_{\Et})^{m}.$ A consequence 
of this compatibility is that if  the endomorphism is an automorphism $\sigma$, then 
 the dynamical evaluation map $f_{\sigma}: S^d \to S^d$ is a bijection, because
$f_{\sigma} \circ f_{\sigma^{-1}} = f_{id}$,  the identity map. 
A second consequence 
is the identity
\begin{equation}
J(\Et_2 \circ \Et_1; {\bf s})= J(\Et_2; f_{\Et_1}({\bf s})) J(\Et_1; {\bf s}),  \quad  {\bf s} \in S
\end{equation}
in  the ring $M_{d \times d}(S)$.

The dynamical evaluation map $f_{\tau}$ is analogous to a map
acting an affine variety as in Theorem \ref{thm:1}.
In general it  
is a nonlinear map, and it usually  does not 
respect either addition or coordinatewise
 multiplication on $S^d$, i.e. one may have 
 $f_{\Et}(\bs_1 + \bs_2) \ne f_{\Et}(\bs_1) + f_{\Et}(\bs_2)$
and $f_{\Et}(\bs_1 \bs_2) \ne f_{\Et}(\bs_1) f_{\Et}(\bs_2)$.
Furthermore one may have $f_{\Et}({\bf 0}) \ne {\bf 0}$, where ${\bf 0}:=(0,0, .., 0) \in S^d.$

The above  definitions  are stated for endomorphisms $\tau$  
but in the remainder of the paper we will restrict to the case 
of automorphisms $\sigma$, unless specified otherwise.
 In  the Appendix to this paper we prove a result that clarifies the
  differences between the geometric and algebraic view of
dynamics of endomorphisms $\tau$ acting on a variety (Proposition \ref{prop:bugs}).

%
%
%
%
\subsection{$p$-adic approximate fixed points for automorphisms}\label{sec22}
\setcounter{equation}{0}

Now we specialize to the case that the  integral domain $S$ is a $\ZZ_p$-algebra.
We consider an  automorphism $\sigma$ of $R = S[x_1, \ldots , x_d]$.
and study the dynamical evaluation map $f_{\sigma}$.
The following lemma asserts the existence of approximate fixed points
$(\bmod \, p)$  with invertible Jacobian matrices for dynamical evaluation maps of 
polynomial automorphisms.\smallskip


\begin{lem} ~\label{lem: j}
Let $p$ be a prime, let $S$ be a $\bZ$-algebra 
that is finitely generated as a $\bZ$-module and let
$ {\sigma}=(F_1,\ldots ,F_d):\bZ[x_1, \ldots, x_d]\rightarrow 
\bZ[x_1, \ldots, x_d]$ be an $\bZ$-algebra automorphism.  Then there is an integer $m$
such that for every point $\bs_0 = (s_1,\ldots ,s_d)\in S^{d}$ with the property
that  the dynamical evaluation  map $f_{\sigma}: S^d \to S^d$
at 
$\bs_0$ satisfies
\[
f_{\sigma}(s_1,\ldots ,s_d)\equiv (s_1,\ldots ,s_d) ~~(\bmod~pS),
\]
the Jacobian $J(\sigma^m; \bx)$ evaluated at
$(x_1, x_2, \ldots, x_d)= (s_1, \ldots, s_d)$ 
is the identity matrix $(\bmod \, pS)$; that is,
$\ev_{\bs_{0}}(J(\sigma^m; \bx))\equiv \,I (\bmod~pS)$.
\end{lem}

 \noindent \begin{proof} The ring $S/pS$ is a finite 
ring since $S$ is a finitely generated $\bZ$-module.  Furthermore,
$J(\sigma; {\bf s})$ has inverse mod $p$ given by $J(\sigma^{-1}; {\bf s})$ and hence it is invertible mod $p$.  Take 
 $m$ to be the order of ${\rm GL}_d(S/pS)$. 
We let
$$
J(\sigma; {\bf s}) := \ev_{\bf s}(J(\sigma; \bx)) \in M_{d \times d}(S)
$$ 
denote  the Jacobian matrix of a 
polynomial map  $\Es$ with its entries evaluated at the point 
${\bf s}\in S^d$.  
Since $\Es$ is a $\bZ$-algebra automorphism,
we have for $\bs_1, \bs_2 \in S^d$
that 
\[
\bs_1 \equiv \bs_2 ~(\bmod~pS) \Rightarrow
f_{\sigma}(\bs_1) \equiv f_{\sigma}(\bs_2)~~(\bmod~pS).
\]
It follows  that the quotient map $\bar{f}_{\sigma}: (S/pS)^d \to (S/pS)^d$
is well-defined.
Similarly, since  $J(\sigma; \bx)$ has entries given by 
polynomials with coefficients in $\bZ$, then
\[ 
\bs_1 \equiv \bs_2 ~(\bmod~ pS) \Rightarrow
J(\sigma; \bs_1) \equiv J(\sigma; \bs_2) ~(\bmod~pS),
\]
where $\equiv$ is taken entry-wise.

Then for any point 
${\bf s}=(s_1,\ldots ,s_d)\in S^d$ we
have 
\begin{eqnarray*}
J(\sigma^m;{\bf s}) & = & 
J(\sigma;f_{\sigma^{m-1}}({\bf s}))J(\sigma;f_{\sigma^{m-2}}({\bf s}))\cdots
J(\sigma;{\bf s}) \\
& = & 
J(\sigma;(f_{\sigma})^{m-1}({\bf s}))
J(\sigma;(f_{\sigma})^{m-2}({\bf s}))\cdots
J(\sigma;{\bf s}).
\end{eqnarray*}
Hence if $\bs_0$ is a fixed point of the quotient map $\bar{f}_{\sigma}$,
then
\[
(f_{\sigma})^j (\bs_0) \equiv \bs_0 ~(\bmod ~pS),
\]
for all $j \ge 1$, so that 
\[
J(\sigma;(f_{\sigma})^j({\bf s_0}))\equiv J(\sigma;{\bf s_0})~(\bmod ~pS).
\]
Substituting these in the formula above yields
\[
J(\sigma^m;{\bf s_0})\equiv J(\sigma;{\bf s_0})^m ~(\bmod~ pS).
\]  
By our choice of $m$,
$M :=J(\sigma^m;{\bf s}_0)$ is congruent to the identity mod $pS$,
as required. $~~~$
 \end{proof} 
\vskip 1mm
\noindent


\begin{rem} \label{rem22}
{\em 
The argument  of Lemma\, \ref{lem: j} 
extends to $\ZZ_p$-algebra endomorphisms $\tau$, but  yields only the weaker conclusion:}
  there exists  a fixed point $\bs_0$ of the evaluation function $f_{\tau} ~(\bmod~ pS)$
at which the Jacobian $\ev_{\bs_{0}}(J(\tau^m; \bx))\equiv M (\bmod~pS)$
where $M^2=M$  is an idempotent  in $M_{d \times d}(S/pS)$.
\end{rem}

%
%
%
%
\subsection{$p$-adic reduction lemmas}\label{sec23}

 We start with an embedding theorem 
due to Lech \cite[\S 4--5]{Le53}.
His result can be regarded as a
$p$-adic analogue of the Lefschetz principle.\medskip

\begin{lem}~\label{le31}  
Let $K$
be a finitely generated
extension of $\mathbb{Q}$ 
and
let $\mathcal{S}$ be a finite subset of
$K$.  Then there exist infinitely many primes $p$ such that 
$K$ embeds in $\mathbb{Q}_p$; moreover, for all but
finitely many of these primes 
every nonzero element of $\mathcal{S}$ 
is sent to a unit in $\mathbb{Z}_p$.
\end{lem}

\begin{proof} This is shown  in  \cite[Lemma 3.1]{Be06}. 
The  Chebotarev density theorem is used
to show there exists  a positive density set of primes $p$ having the
required  property. $~~~$
\end{proof}
\vskip 2mm

Strassman's theorem \cite{St28}  asserts that  if a power series 
$f\in \mathbb{Q}_p[[z]]$  converges in the closed $p$-adic unit disc
$$
B_{\QQ_p}(0; 1) := \{ z \in \QQ_p:~ |z|_p \le 1 \} = \ZZ_p.
$$
and has infinitely many zeros in this disc, then it is  identically zero.
We will use the following variant of Strassman's theorem.\\
\begin{thm}~\label{thm: Strassman} {\rm (Extended Strassman's Theorem)}
Let $p$ be a prime and let $R$ be a finite-dimensional $\bQ$-algebra.  
Suppose that the formal power series $f(z)\in R[[z]]$ is absolutely convergent for 
all $z\in \ZZ_p$ and has 
infinitely many zeros in $\ZZ_p$.  Then $f(z)$ is identically zero.
\end{thm}
\begin{proof}

Let $n$ denote the dimension of $R$ as a $\bQ$-vector space.  We write $$R= \bQ v_1+\bQ v_2+\cdots +\bQ v_n,$$
with $v_1=1$.
Then 
$$
f(z) =v_1f_1(z) + \cdots+ v_n f_n(z) 
$$
with each $f_i(z) = \sum_{k=0}^{\infty} c_k z^k \in \bQ[[z]]$,
converging for all $z \in \ZZ_p$. Linear independence of
the $v_i$ over $\QQ_p$ implies that  $f(z)=0$ requires
that each $f_i(z)=0$ separately. Since $f(z)$
has infinitely many zeros on $z \in \ZZ_p$, 
by Strassman's theorem (\cite{St28}, cf. Cassels \cite[Theorem 4.1]{Ca86})
applied to $f_i(z)$, 
each $f_i(z)$ is identically zero.   
Thus $f(z)$ is identically zero. $~~~$
\end{proof}

Finally, we will need   a finiteness result for reduction (mod $p$) in
a $\ZZ_p$-module.\\

\begin{lem}~\label{le22} 
Let $p$ be prime and let $\mathcal{M}$ be a $\bZ$-module that is isomorphic to a submodule of $\bQ^d$ for some natural number $d$, 
then $\mathcal{M}/p\mathcal{M}$ is a finite-dimensional $\mathbb{Z}/p\mathbb{Z}$-vector space.
\end{lem}

\noindent \begin{proof}  
Note that $\mathcal{M}/p\mathcal{M}$ is a $\mathbb{Z}/p\mathbb{Z}$-vector space.  We claim that its dimension is at most $d$.  To see this, let $\theta_1,\ldots ,\theta_{d+1}$ be elements of $\mathcal{M}$.   
If we regard $\mathcal{M}$ as a submodule of
$\bQ^d$, then we see that $0$ is a non-trivial $\bQ$-linear 
combination of the images of $\theta_1,\ldots ,\theta_{d+1}$ in $\bQ^d$.   
Clearing denomiators, we see that there exist 
$a_1,\ldots ,a_{d+1}\in \bZ$, not all of which are in $p\bZ$, such that 
$\sum_{i=1}^{d+1} a_i \theta_i=0$.  Reducing mod $p\mathcal{M}$, 
we see that the images of
$\theta_1,\ldots ,\theta_{d+1}$ in $\mathcal{M}/p\mathcal{M}$ are linearly dependent 
over $\mathbb{Z}/p\mathbb{Z}$ and thus any set of size $d+1$ in $\mathcal{M}/p\mathcal{M}$ is 
linearly dependent.  Hence $\mathcal{M}/p\mathcal{M}$ is at most $d$-dimensional and in particular it is a finite-dimensional $\ZZ/p\ZZ$-vector space.\end{proof}

%
%
%
%
\section{Generalized $p$-adic Analytic Arc Theorem}\label{sec3}
\setcounter{equation}{0}

In this section we prove a generalization of
the $p$-adic analytic arc theorem given in \cite[Theorem 1.1]{Be06b}
that applies to a larger class of rings.
The generalization applies
to any  $\ZZ_p$-algebra $S$ that is finitely generated and torsion-free as a $\ZZ_p$-module with the added hypothesis that $p \ge 5$.
Examples given in \cite{Be06} show this theorem does  not hold for $p=2$;  
the case  $p=3$ of the result below is currently open.

%
\begin{thm} {\em (Generalized $p$-adic analytic arc theorem)} \label{thAA}
Let $p\ge 5$ be prime, and suppose \\
$\sigma=(H_1, H_2, \ldots, H_d)$ 
 is a polynomial  automorphism of ${\mathbb Z}_p[x_1,\ldots ,x_d]$.
Let  $S$ be 
a $\bZ$-algebra that is finitely generated and torsion-free as a $\bZ$-module,
and let $f_{\sigma}: S^d \to S^d$ denote the induced 
dynamical evaluation map.
Suppose that an initial value
$\bs_0 =(s_1,\ldots ,s_d)\in S^{d}$ satisfies the following two conditions: 
\begin{enumerate}
\item[(i)]  $H_i(s_1,\ldots ,s_d)\equiv s_i~(\bmod ~ pS)$ for $0\le i\le d$;
\item[(ii)] The Jacobian matrix $M=J(\sigma; \bx) |_{\bx=\bs_0}$ 
evaluated at $\bx=\bs_0$ is the identity matrix $(\bmod ~pS)$.
\end{enumerate}
Then there exist power series 
$f_1(z),\ldots ,f_d(z) \in (S\otimes_{\bZ} \bQ)[[z]]$ 
which converge for all  $z \in \ZZ_p$  and
which satisfy: 
\begin{enumerate}
\item $f_i(z+1) \ = \
H_i(f_1(z),\ldots ,f_d(z))$, \,for $1\le i\le d$;\\
\item $f_i(0)= s_i$, \,for $1\le i\le d$.
\end{enumerate}
\end{thm}

\begin{rem}\label{rem36a}
{\em The  associated {\em analytic arc}  in $S^d$ constructed in Theorem ~\ref{thAA}  is 
$$
\sC= \sC( \sigma, \bs_0) := \{ (f_1(z), \ldots , f_d(z)): z \in \ZZ_p\} \subset S^d.
$$
The result implies that the  arc  $\sC$ contains all the iterates $\{ f_{\sigma}^{(m)} (\bs_0) : m \ge 0\}$ of the initial value $\bs_0$,
since (1) shows $ f_{\sigma}^{(m)} (\bs_0)= (f_1(m), \ldots , f_d(m))$ holds.
IT covers cases such as $S= \ZZ_5$ with $\sigma(x) = x+5$ on $\ZZ_5[x]$, 
where one may take $\bs_0=0$, and  nevertheless
 the induced dynamical evaluation map $f_{\sigma}$ has no fixed points.
 The proof  follows the  basic plan of the result in \cite{Be06},  but involves more complications.
It shows  that  conditions (i), (ii)
  imply that $f_{\sigma}^{(p^m)}(\bs_0)$ converges rapidly to $s_0$ as $m\to \infty$,
and uses this  to show analyticity of the constructed maps $f_1(z),\ldots ,f_d(z)$.} 
\end{rem}
\begin{rem}\label{rem36b}
{\em  Theorem \ref{thAA} cannot be extended to cover all general algebra endomorphisms of $\ZZ_p[x_1, ..., x_d]$.  
For general endomorphisms there may be no point where condition (b) holds, and  
Remark \ref{rem22} asserts  one can only guarantee there exists a point where 
the Jacobian, when evaluated  $(\bmod ~pS)$, is an idempotent matrix.    The worst case is where the
Jacobian vanishes identically $(\bmod ~pS)$. Consider 
 $d=1$ with $H_1(x) = x^{10}$, where we take $p=5$ and $s_0=5$, a case where
 the Jacobian vanishes identically.  
 In this case, we have that $\sigma^n(s_0)=5^{{10}^n}$ for $n\ge 0$.  Suppose that it were 
 possible to find an arithmetic progression $\{an+b\}_{n\in \mathbb{Z}}$ and a $5$-adic analytic map 
 $f:\mathbb{Z}_5\to \mathbb{Z}_5$ with the property that $f(n)=\sigma^{an+b}(s_0)$ for all integers $n$.  
 Then we would necessarily have $f(5^k)\to 0$ as $k\to \infty$.  But $f(0)=\sigma^b(5)\neq 0$ and so it is impossible
 to find even a  continuous map $f:\mathbb{Z}_5\to \mathbb{Z}_5$ that has the property (1).}
\end{rem}

To begin the proof, we  first  introduce a ring of polynomials which takes
``integral'' values when evaluated at
values in $\ZZ_p$.\smallskip

\begin{defi}~\label{de44}
{\em 
Given a prime $p$ and a $\bZ$-algebra $S$ that is 
finitely generated and torsion-free as a $\bZ$-module,  let
 $\sRS$ denote
the subring of the polynomial ring
$(S\otimes_{\bZ} {\mathbb Q}_p)[z]$ that consists
of those polynomials which under evaluation at each $z \in \ZZ_p$
belong to $S$.
}
\end{defi}
\smallskip

The $\ZZ_p$-algebra $\sRS$ offers a crucial advantage over the polynomial ring
$S\otimes_{\bZ}
{\mathbb Q}[z]$, which is 
that the following closure 
property holds: if $Q(z)\in \sRS$ then the equation 
\begin{equation}\label{301}
F(z+1)-F(z)=Q(z)
\end{equation}
 has a solution with $F(z)\in \sRS$.  We use this in proving our general analytic arc theorem (Theorem \ref{thAA}).  
 We note however that the  ring $\sRS$ has some pathological properties.
 It  is a non-Noetherian ring,
 by a general criterion of  Cahen and Chabert \cite[Prop. VI.2.4]{CC97}.
 (In fact, for  $S= \ZZ_p$, the ring $\sRS = \ZZ_p[ f_n(z) = \frac{z(z-1)\cdot(z-(n-1))}{n!}: n \ge 0]$
 as a linear space over $\ZZ_p$---see Lemma \ref{le44} below. An example of an infinite ascending chain of ideals in $\sRS$
 that does not stabilize  is:
 $(z), (z, f_p(z)), (z, f_p(z), f_{p^2}(z)), \ldots$.)
 Secondly,  it can be shown that $\sRS$  is not of finite type over $\ZZ_p$, i.e., it
 is not finitely generated as a $\ZZ_p$-algebra.
  Finally, $\sRS$ is not a UFD, in general: 
 for $S=\ZZ_2$, then $u:=z(z-1)/2$, $v:=(z-2)(z-3)/2$, $u':=z(z-3)/2$, $v':=(z-1)(z-2)/2$ are irreducible elements of $\sRS$ and $uv=u'v'$; there are similar examples for $S=\ZZ_p$, $p \ge 3$.\smallskip

\begin{lem}~\label{le44}
Given a prime $p$ and a 
$\bZ$-algebra $S$ that is finitely generated and 
torsion-free as a $\bZ$-module, the ring $\sRS$
is given by 
\begin{equation}
\sRS \ = \ 
\Bigg\{ f(z)=
\sum_{i=0}^m b_i {z\choose i}~:~ m\ge 0~{\rm and}~b_i\in S\Bigg\}.
\end{equation}
\end{lem}
\begin{proof} In the special case $S=\ZZ_p$ the lemma asserts that  
the polynomials ${z\choose k}= \frac{z(z-1) \cdots (z-k+1)}{k!}$
for $k \ge 0$ are a basis  of the polynomials in $\QQ_p[z]$ that map $\ZZ_p$
into itself. This is a theorem of Mahler \cite[p. 49--50]{Mah81}.
For the general case, since  $S$ is finitely generated and torsion-free,
$S \otimes_{\bZ} \QQ_p [z]$ is a finite direct sum 
of copies of $\QQ_p[z]$. The result then follows from the special case
applied to each copy separately.
\end{proof}


\begin{lem} \label{lem: matrix}
Let $p$ be prime and let $S$ be a  $\bZ$-algebra 
that is finitely generated and torsion-free as a $\bZ$-module.
Suppose 
$Q_1(z), \ldots, Q_{d}(z) \in \sRS$  are given polynomials of degree at most $N$.
Then there exists a solution 
$[h_1(x),\ldots ,h_d(x)]^{\rm T}\in \sRS^n$ 
to the equation
\begin{equation} \label{eq: mateq} 
\left[\begin{array}{c} h_{1}(z+1) \\ \vdots \\ h_{n}(z+1)\end{array}\right]
\ \equiv \ \left[\begin{array}{c} h_{1}(z) \\ \vdots \\ h_{d}(z)\end{array}
\right] -
\left[\begin{array}{c} Q_{1}(z) \\ \vdots \\ Q_{d}(z)\end{array}\right]
\ ({\rm mod}~p\sRS)\end{equation}
such that $h_1(0)=\cdots =h_d(0)=0$ and $h_1,\ldots ,h_d$ have degree at most $N+1$.
\end{lem}

\noindent \begin{proof} By assumption, each $Q_i (z) \in \sRS$ is
of degree at most $N$, and so 
\[ 
Q_{i}(z) \ = \ \sum_{k=0}^{N} c_{i,k} {z\choose k}
\]
with each $c_{i,k} \in S$.
We define 
\begin{equation} \label{eq: hi0}
h_{i}(z) \ := \ - \sum_{k=0}^N c_{j,k} {z\choose k+1},
\end{equation}
which implies $h_{i} \in \sRS$ and is of degree at most $N+1$.
It is easy to check that this gives a solution 
to equation (\ref{eq: mateq}), using the identity
\[
 {z+1\choose k+1}-{z \choose k+1} = {z \choose k}.
\]
Furthermore, $h_i(0)=0$ for $1\le i\le d$. 
\end{proof} 

\vskip 2mm
To create analytic maps in the modified version of the $p$-adic analytic arc lemma, we will use
the following lemma about subalgebras of $\sRS$.\smallskip


%

\begin{lem}\label{le34}
 Let $p$ be a prime and let $S$ be a  $\bZ$-algebra 
that is finitely generated and torsion-free as a $\bZ$-module.  Let $N$ be a natural number, and let
\[ S_N \ = \ \Bigg\{ c + \sum_{i=1}^N p^i h_i(z)~|~c\in S, h_i(z)\in \sRS, {\rm deg}(h_i)\le 2i-1\Bigg\}\]
and
\[ T_N \ = \ S_N+\Bigg\{ c + \sum_{i=1}^M p^i h_i(z)~|~M\ge 1, c\in S, h_i(z)\in \sRS, {\rm deg}(h_i)\le 2i-2\Bigg\}.\]
Then the $\bZ$-subalgebra of $\sRS$ generated by $S_N$ is contained in $T_N$.
\label{lem: SN}
\end{lem}
\noindent \begin{proof} Since $S_N$ and $T_N$ are both closed under addition and $S_N\subseteq T_N$, it is sufficient to show that $S_NT_N\subseteq T_N$.
To do this, suppose
\[H(z) \ = \ c + \sum_{i=1}^N p^i h_i(z) \ \in \ S_N\] and
\[G(z) \ = \ d + \sum_{i=1}^M p^i g_i(z) \ \in \ T_N,\]
where $c,d\in S$ and $h_i(z),g_i(z)\in \sRS$ with ${\rm deg}(g_i)\le 2i-2$ for $i>N$ and ${\rm deg}(g_i),{\rm deg}(h_i)\le 2i-1$ for $i\le N$.
We must show that $H(z)G(z)\in T_N$.  Notice that
\[
H(z)G(z) \ = \ cG(z)+dH(z) -cd + \sum_{i=1}^N p^i h_i(z)\sum_{j=1}^M p^j g_j(z). \]
Then $cG(z), dH(z)$, and $cd$ are all in $T_N$ and since $T_N$ is closed under addition, it is sufficient to show that
\[\sum_{i=1}^N p^i h_i(z)\sum_{j=1}^M p^j g_j(z)  \ = \ \sum_{k=2}^{N+M} p^k \sum_{i=1}^{k-1}g_i(z)h_{k-i}(z)\]
is in $T_N$.  But since $g_i(z)$ has degree at most $2i-1$ and $h_{k-i}(z)$ has degree at most $2(k-i)-1$, we see that 
$$\sum_{i=1}^{k-1} g_i(z)h_{k-i}(z)$$ has degree at most $2k-2$.  It follows that
\[\sum_{i=1}^N p^i h_i(z)\sum_{j=1}^M p^j g_j(z)  \ \in \ T_N.\]
The result follows. $~~~$ \end{proof}

\noindent \begin{proof}[Proof of Theorem \ref{thAA}] 
We  construct $(f_1(z), \ldots, f_{d}(z))$ by successive approximation 
$(\bmod~ p^{j} \sRS)$.  The approximations will be denoted $g_{i, j}(z)$ for
$1 \le i \le d$. We initialize with 
$$
g_{i,0}(z) := s_i   ~~~~~\mbox{for}~~ 1\le i \le d.
$$

We prove by induction on $j$ that one can recursively pick 
$$
g_{i,j}(z) :=s_i+\sum_{k=1}^{j} p^{k}h_{i,k}(z),
$$
such  polynomials $h_{i,j}(z) \in \sRS$  ($1\le i\le d$)  satisfy
the three conditions:
\begin{enumerate}
\item $h_{i,j}(0)=0$ for $1\le i\le d$;
\item $h_{i,j}(z)$ has degree at most $2j-1$ for  $1\le i\le d$;
\item there holds 
$$ 
g_{i,j}(z+1)\equiv H_i(g_{1,j}(z),\ldots ,g_{d,j}(z)) ~(\bmod ~
p^{j+1}\sRS ).
$$
\end{enumerate}
  The base case of the induction is $j=0.$ Conditions (1) and (2) are
  vacuous, and (3) holds using hypothesis (i), observing that  $\sRS \cap S=S$.

  Let $j \ge 1$ and assume that we have defined $h_{i,k}$ for $0\le i\le d$ 
and $k <j $
so that conditions (1)--(3) hold.  Our object is now to construct
\[
g_{i, j+1}(z) := g_{i,j}(z) + p^j h_{i, j}(z),
\]
in which polynomials $h_{i, j}(z) \in \sRS$ are to
be determined, so that conditions (1)--(3) hold.
By assumption
\begin{equation} \label{eq: gH}
g_{i,j-1}(z+1)-H_i(g_{1,j-1}(z),\ldots ,g_{d,j-1}(z))=p^jQ_{i,j}(z),
\end{equation}
with
$Q_{i,j}\in \sRS$
for $1\le i\le d$.
Using the notation of the statement of Lemma \ref{lem: SN}, we see that conditions (2) and (3) show that
$g_{1,j-1}(z),\ldots ,g_{d,j-1}(z)$ are in $S_{j-1}$.
Thus by Lemma \ref{lem: SN} we see that
\[ p^j Q_{i,j}(z)\ = \ g_{i,j-1}(z+1)-H_i(g_{1,j-1}(z),\ldots ,g_{d,j-1}(z))\]
is in the algebra generated by $S_{j-1}$ and hence is in $T_{j-1}$.  It follows that we can write
$p^j Q_{i,j}(z)= c_{i,j}+ \sum_{k=1}^M p^k q_{i,j,k}(z)$ for some $c_{i,j}\in  S$ and polynomials $q_{i,j,k}(z)\in \sRS$ such that ${\rm deg}(q_{i,j,k})\le 2k-1$ for $k\le j-1$ and ${\rm deg}(q_{i,j,k})\le 2k-2$ for $k\ge j$.  Consequently,
$p^jQ_{i,j}(z)$ is equivalent modulo $p^{j+1}R$ to the polynomial
$$ c_{i,j}+ \sum_{k=1}^{j} p^k q_{i,j,k}(z),$$
a polynomial of degree at most $2j-2$.   Hence $Q_{i,j}(z)$ is congruent to a polynomial in $\sRS$ of degree at most $2j-2$ mod $p\sRS$.
To satisfy property
(3) for $j$ it  is sufficient to find 
$\{h_{i, j}(z) \in \sRS : 1 \le i \le d\}$ such
that 
\[
g_{i,j-1}(z+1)+p^j h_{i,j}(z+1)  - H_i(g_{1,j-1}(z)+p^jh_{1,j}(z),\ldots ,
g_{d,j-1}(z)+p^jh_{d,j}(z))\]
is in  $p^{j+1}\sRS$
for $1\le i\le d$.
This expression becomes
\[
p^jQ_{i,j}(z) + p^j h_{i,j}(z+1)
 - p^j \sum_{\ell=1}^d h_{\ell,j}(z) \frac{\partial H_{i}}
{\partial x_{\ell}} (x_1, \ldots, x_d)
\Big|_{x_{1}=g_{1,j}(z),\ldots, x_d=
g_{d,j}(z)} \]
modulo $p^{j+1}\sRS$. 
It therefore suffices  to solve 
for $1 \le i \le d$  the system
\begin{equation} \label{eq: Q}
Q_{i,j}(z) + h_{i,j}(z+1)
 - \sum_{\ell=1}^d h_{\ell,j}(z) \frac{\partial H_{i}}
{\partial x_{\ell}}(x_1, \ldots, x_d) \Big|_{x_{1}=g_{1,j}(z),\ldots, x_d=
g_{d,j}(z)} 
\end{equation}
$ (\bmod ~p\sRS)$, where we may assume that $Q_{i,j}$ is of degree at most $2j-2$.
Now consider  the Jacobian matrix $M^{(j)}(z) \in M_{d \times d}(\sRS)$
with polynomial entries
\[
M^{(j)}(z)_{i\ell} := \frac{\partial H_{i}}
{\partial x_{\ell}}(x_1, \ldots, x_d) \Big|_{x_{1}=g_{1,j}(z),\ldots, x_d=
g_{d,j}(z)}.
\]
Property (1) for $j$ yields for $1 \le i \le d$ that 
\[ 
g_{i, j}(z) \equiv s_i\, (\bmod~ p \sRS).
\]
It follows that
\[
M^{(j)}(z) \equiv J(\sigma; \bs_0) (\bmod~ p \sRS).
\]
By hypothesis (ii) the matrix $M=  J(\sigma; \bs_0) \in M_{d\times d}(S)$
is congruent to the identity $(\bmod~pS)$, and we have
\[
M^{(j)}(z) \equiv M (\bmod~ p \sRS).
\]
Now equation (\ref{eq: Q}) can be rewritten in the form 
\[ 
\left[\begin{array}{c} h_{1,j}(z+1) \\ \vdots \\ h_{d,j}(z+1)\end{array}
\right]
\equiv M\left[\begin{array}{c} h_{1,j}(z) \\ \vdots \\ h_{d,j}(z)\end{array}
\right] - 
\left[\begin{array}{c} Q_{1,j}(z) \\ \vdots \\ Q_{d,j}(z)\end{array}
\right]
(\bmod~ p\sRS).
\]
The hypotheses of Lemma \ref{lem: matrix} are satisfied, so
we conclude that 
there exists a  solution $[h_{1,j}(z), \ldots, h_{d, j}(z)] \in \sRS^d$ with $h_{i,j}(0)=0$ for $1\le i\le d$ and
$h_i(z)$ of degree at most $2j-1$.
Thus conditions (1)--(3) are satisfied for $j$, completing the induction
step.

We now set
\[
f_i(z) := s_i + \sum_{j=1}^{\infty} p^j h_{i, j}(z).
\]
Now each $h_{i,j}(z) \in \sRS$ is of degree at most $2j-1$ and hence
\[
h_{i,j}(z) = \sum_{k=0}^{2j-1}  c_{i,j,k} {z \choose k},
\]
with $c_{i,j,k} \in S$.
We find that
\begin{eqnarray}
f_i(z) &=& s_i + \sum_{j=1}^{\infty} p^j 
\left( \sum_{k=0}^{N_j} c_{i,j,k} {z \choose k}\right) \nonumber\\
&=& s_0 + \sum_{k=0}^{\infty} b_{i,k}  {z \choose k} \label{eq475}
\end{eqnarray}
in which
\[
b_{i,k} := \sum_{j=1}^{\infty} p^j c_{i,j,k}
\]
is absolutely convergent $p$-adically, since each $c_{i,j,k} \in S$.
To show the series (\ref{eq475}) converges to an analytic map on
 $z \in \ZZ_p$, we must establish that $|b_{i,k}|_{p}/|k!|_p \to 0$
as $k \to \infty$, i.e. that for any $j >0$ one has $b_{i,k}/k! \in p^j \bZ$
for all sufficiently large $k$.
To do this, we note that $c_{i,j,k}=0$ if $j<(k+1)/2$.  Hence
\[ b_{i,k} := \sum_{j\ge (k+1)/2} p^j c_{i,j,k}.\]
It follows that $|b_{i,k}|_p<p^{-(k+1)/2}$. Since $1/|k!|_p<p^{k/(p-1)}$, we see 
that 
$b_{i,k}/k!\rightarrow 0$ since $p>3$.  Hence $f_1,\ldots ,f_d$ are analytic maps on $\bZ$.
 
The argument above also showed that
\[ 
f_i(z) \equiv g_{i,j}(z)~(\bmod~p^{j}\sRS).
\]
It then follows from property (3) above that
\[
f_i(z+1) \equiv H_i( f_1(z), \ldots, f_d(z)) (\bmod~ p^j \sRS)
\]
Since this holds for all $j \ge 1$, we conclude
that
\[
f_i(z+1) = H_i( f_1(z), \ldots, f_d(z)).
\]
This establishes (1).
Finally, we have
\[ 
f_{i}(0) = s_i + \sum_{j=1}^{\infty} p^j h_{i,j}(0) = s_i,
\]
which establishes (2). $~~~$ \end{proof} 
\smallskip

\section{Polynomial Ring Case}\label{sec4}
\setcounter{equation}{0}

In this section, we consider automorphisms of a polynomial ring over a field.
We prove the following theorem.  

\begin{thm}
\label{thm: polycase}
Let $K$ be a field of characteristic zero and let $A=K[x_1,\ldots ,x_d]$.  If $\sigma: A\to A$ is a $K$-algebra automorphism 
and $I$ and $J$ are ideals of $A$, then $$\{n\in \mathbb{Z}~:~\sigma^n(I)\supseteq J\}$$ 
has the two-sided SML property.
\end{thm}
In Section \ref{sec5} we  will deduce the main   Theorem \ref{thm:main} from this result.  
We establish Theorem \ref{thm: polycase} using a reduction to  the ring $\ZZ_p[x_1, \ldots, x_d]$, given 
below (Theorem  \ref{th41}). 
A key ingredient in the proof of Theorem \ref{th41} will be the $p$-adic analytic arc theorem proved earlier, see
Proposition \ref{pr47}.

In restricting ideals from a ring with $\QQ_p$-coefficients to one with $\ZZ_p$-coefficients, we only need to
consider the following special subclass of ideals in $\ZZ_p[x_1, x_2, \ldots, x_d]$, as follows.


\begin{defi}\label{def42}
{\em
Let $p$ be a prime and let $A$ be a $\mathbb{Z}_p$-algebra.  We say that a proper ideal $I$ of $A$ is $p$-\emph{reduced}
 if whenever $x\in A$ has the property that $px\in I$, we necessarily have $x\in I$.  }
\end{defi}

 The condition that an ideal $I$ of a $\mathbb{Z}_p$-algebra $A$ be $p$-reduced is
imposed to ensure that the quotient ring $A/I$ is torsion-free as a $\mathbb{Z}_p$-module.  

\begin{prop} \label{pr47}
Let $p$ be prime, and let
$I$ and $J$ be two $p$-reduced ideals in $\bZ[x_1,\ldots ,x_d]$.
Suppose that 
$(\bZ[x_1,\ldots ,x_d]/I)\otimes_{\bZ} \bQ$ is 
finite dimensional as a $\bQ$-vector space.
If $\sigma$ is a $\bZ$-algebra automorphism  of
$\bZ[x_1,\ldots ,x_d]$, then
 the set
$\{ n \in \mathbb{Z}~|~
\sigma^{n}(I)\supseteq J\}$
has the two-sided SML property.
\end{prop}

\noindent \begin{proof} 
For a polynomial automorphism  $\sigma^n(I)\supseteq J$ if and only if $\sigma^{-n}(J)\subseteq I$.  Thus,
replacing $\sigma$ with $\sigma^{-1}$,  it is sufficient to show that 
$$\{ n\in \mathbb{Z}~:~\sigma^n(J)\subseteq I\}$$ is a finite union of  complete
arithmetic
progressions and a finite set.
Let $S:=\bZ[x_1,\ldots ,x_d]/I$. 
The ring $S$ is a finitely generated and torsion-free $\ZZ$-algebra,  since $I$ is $p$-reduced.
Write 
$\sigma=(F_1,\ldots ,F_d)\in \bZ[x_1,\ldots ,x_d]^d$.
The  automorphism $\sigma$ induces the dynamical evaluation map  $f_{\sigma}:S^d \to S^d$, given by 
$(s_1,\ldots ,s_d)\mapsto 
(F_1(s_1,\ldots ,s_d),\ldots ,F_d(s_1,\ldots ,s_d))$, for $s_i \in S$,
which is a (nonlinear) bijection on $S^d$.
Since $S$ is a torsion-free $\bZ$-algebra and since, by hypothesis,
$S\otimes_{\bZ} \bQ$ is finite-dimensional,
we have  $S^d/pS^d$ is  a finite ring.  Although $f_{\sigma}$ is  nonlinear, we have 
$$
f_{\sigma}(s+ pS^d)\subseteq f_{\sigma}(s)+ pS^d,
$$ 
 hence $f_{\sigma}$ induces a well-defined bijective map of $S^d/pS^d$. 
 Since $S/pS$ is finite, 
there exists a positive integer $a$
such that $f_{\sigma}^a({\bf s})\equiv {\bf s}~(\bmod~ pS^d)$ 
for all ${\bf s}\in S^d$.

Now set $\tauu = \sigma^{ma}$, where $m$ is chosen as in the statement of Lemma \ref{lem: j}, 
and write
$\tauu=(H_1,\ldots ,H_d)$.  
Then for any ${\bf s}\in S^d$, 
we have $f_{\tauu}({\bf s})\equiv {\bf s} ~(\bmod ~pS^d)$ and by Lemma \ref{lem: j},
the 
Jacobian matrix of $\tauu$ is the identity matrix mod $pS$ 
when evaluated at any point in $S^d$.  
For  any polynomial $P(x_1,\ldots ,x_d)\in \bZ[x_1,\ldots ,x_d]$
and for any $k$ with $0 \le k \le ma-1$, the fact that $\sigma$ and $\tauu$ are endmorphisms gives
(via \eqref{201})
\[
\tauu^n(\sigma^{k}(P(x_1,\ldots ,x_d)))\in I~~
\mbox{if and only if}~~
P\circ f_{\tauu^n}\circ f_{\sigma^{k}}(s_1,\ldots ,s_d)=0,
\]
where $s_i :=x_i+I$. We now treat each $k$ separately.
By construction, 
\[
f_{\tauu}\circ f_{\sigma^{k}}(s_1,\ldots ,s_d)\equiv f_{\sigma^{k}}(s_1,\ldots ,s_d) ~(\bmod ~pS),
\]
 and the Jacobian matrix $J(\tauu,f_{\sigma^{k}}(s_1,\ldots ,s_d))$
 is congruent to the identity mod $pS$.
Now the generalized $p$-adic analytic arc theorem (Theorem \ref{thAA})
applies to show there exist power series 
$$f_1(z) ,\ldots ,f_d(z)\in (S\otimes_{\bZ} \bQ)[[z]]$$
which are analytic on $\ZZ_p$  and 
which satisfy $(f_1(0),\ldots ,f_d(0))=f_{\sigma^{k}}(s_1,\ldots ,s_d)$ and
$$
f_i(z+1) \ = \ H_i(f_1(z),\ldots ,f_d(z))\qquad {\rm for}~1\le i\le d.
$$  
By construction,
$f_{\tauu^n}\circ f_{\sigma^{k}}(s_1,\ldots ,s_d))=(f_1(n),\ldots ,f_d(n))$ 
for all $n\in {\mathbb N}$.
Next,  select a generating set for the ideal $J$, 
$$
J=\langle P_1(x_1,\ldots ,x_d),\ldots ,P_m(x_1,\ldots ,x_d)\rangle.
$$
Then for  $1\le i\le m$, define
$$
g_i(z)\ := \ P_i(f_1(z),\ldots ,f_d(z))\in 
(S\otimes_{\bZ} {\mathbb Q}_p)[[z]].
$$
Since $f_1,\ldots ,f_d$ are analytic on $\ZZ_p$ 
and $P_1,\ldots ,P_m \in \ZZ_p[x_1,\ldots , x_d]$,
each  $g_i(z)$ is
a power series which converges on $\ZZ_p$, for  $1\le i\le d$.
Moreover, 
$g_i(n)=0$ if and only if $\tauu^n\circ \sigma^{k}(P_i(x_1,\ldots ,x_d))\in I$.  
Notice if $g_i(n)=0$ for infinitely many 
integers $n$, then it is identically zero by
Theorem \ref{thm: Strassman}. Thus 
$$
\{ n\in \mathbb{Z}~|~ \tauu^n(\sigma^{k})(J) \subseteq I\}
$$
is either a finite set or is all of $\mathbb{Z}$.
Since this holds for $0 \le k  \le ma_1,$  result now follows. $~~~$ \end{proof}

\noindent 
We must also treat the case that $\left( \ZZ_p[x_1,\ldots ,x_d]/I\right) \otimes \QQ_p$   is
an infinite-dimensional $\QQ_p$-vector space.  This is handled
by a reduction to the preceding proposition. 
To do this reduction, we use the following lemma,
based on an idea of  Amitsur \cite[Lemma 4, p. 41]{Am56}.

\begin{lem} \label{lem: amitsur}
Let $K$ be an uncountable field and let $A$ be a finitely generated 
commutative $K$-algebra.  Suppose that there exists a countable set of ideals 
$\{I_i~|~i\ge 1\}$ of $A$ such that each ideal $L$ of $A$ 
of finite codimension contains one of the $I_i$.  
Then there exists a finite set $j_1,\ldots ,j_d$ such that $$I_{j_1}\cap \cdots\cap  I_{j_d}=(0).$$ 
\end{lem}

\noindent \begin{proof}
 Let $J(A)$ denote the Jacobson radical 
of $A$.  Since $A$ is a finitely generated $K$-algebra, 
$A$ is a Jacobson ring (that is, every prime ideal of $A$ is the intersection of the maximal ideals containing it)
 \cite[Theorem 4.19]{Eis95}.  Thus $J(A)$ is the intersection of the prime ideals of $A$ (that is, it is the nilradical of $A$).
    By the Hilbert basis theorem $A$ is Noetherian and hence some power of the nilradical  of $A$
     is zero \cite[Corollary 7.15]{AMc}.   Hence there is some $m\ge 1$ such that $J(A)^m=(0)$.
Given a maximal ideal $M$ of $A$, note that there is 
some $i$ such that $ I_i \subseteq M^m$.  Let $Q_i$ denote the intersection of all ideals of the form $M^m$, where $M$ is a maximal ideal of $A$ with the property that $M^m\supseteq I_i$.  Similarly, we let $P_i$ denote the intersection of all ideals maximal ideals $M$ with the property that $M^m\supseteq I_i$.  Then $P_i^m=Q_i\supseteq I_i$ for each $i\ge 1$.

Each maximal ideal $M$ of $A$ necessarily contains some $P_i$.  A result of Amitsur 
(cf. \cite[Corollary 4, p. 40]{Am56}) shows that there exist $j_1,\ldots ,j_d$ such that $P_{j_1}\cap P_{j_2}\cap \cdots\cap P_{j_d}$ is contained in
the Jacobson radical $J(A)$ of $A$.  Hence $$I_{j_1}\cap \cdots\cap I_{j_d}\subseteq Q_{j_1}\cap \cdots \cap Q_{j_d}=P_{j_1}^m\cap \cdots \cap P_{j_d}^m\subseteq J(A)^m=(0),$$ 
as required.   \end{proof} \smallskip

\begin{thm} ~\label{th41}
Let $p\ge 5$ be prime, let
$I$ and $J$ be two $p$-reduced ideals in 
$\bZ[x_1,\ldots ,x_d]$ with $I\cap \mathbb{Z}_p=J\cap \mathbb{Z}_p=(0).$
If $\Es$ is  a $\bZ$-algebra automorphism of
$\bZ[x_1,\ldots ,x_d]$,
then the set 
\[
\{ m \in \mathbb{Z}~|~
\Es^{m}(I)\supseteq J\}
\] 
is the union of a finite number of complete doubly-infinite arithmetic
progressions, up to the addition of a finite set.
\end{thm}

\begin{proof}
Let $\mathcal{T}$ denote the collection of subsets of 
$\mathbb{Z}$ which are finite union of complete doubly-infinite arithmetic progressions 
along with a finite set.  Then $\mathcal{T}$ is countable.  
Suppose first that $I\subseteq \bZ[x_1,\ldots ,x_d]$ is a $p$-reduced 
ideal with the property that
$(\bZ[x_1,\ldots ,x_d]/I)\otimes_{\bZ} \bQ$ is 
finite-dimensional as a $\bQ$-vector space.  
Then by Proposition \ref{pr47}, 
we have
\[
\{ n\in \mathbb{Z}~|~\sigma^n(I)\supseteq J \} \ \in \ \mathcal{T}.
\]  
We now treat the general case.  
Given $T\in \mathcal{T}$, let 
$I_T$ denote the intersection of all $p$-reduced ideals 
$L$ that contain $I$ and such that:
\begin{enumerate}
\item $L\cap \mathbb{Z}_p = (0)$;
\item $\left(\bZ[x_1,\ldots ,x_d]/L\right)\otimes_{\bZ} \bQ$ is a 
finite-dimensional $\bQ$-vector space;
\item $T= \{ n\in \mathbb{Z}~|~\sigma^n(L) \supseteq  J \}$.  
\end{enumerate}

Suppose first that $T\in \mathcal{T}$ is such that $I_T$ is a non-empty intersection.
We first note that if $n\in T$, 
then
$\sigma^n(I_T)\supseteq J$; moreover $I_T$ is an intersection of ideals which 
contain $I$ and hence it contains $I$.  If $n\not\in T$,
then $\sigma^n(I)$ cannot contain $J$, since by definition
$\sigma^n(I_T)\not \supseteq J$ and $I\subseteq I_T$.  Thus if $I_T$ is a non-empty intersection, then 
$\{n~|~   \sigma^n(I_T) \supseteq J\}=T$.  

We next note that the collection of ideals
\[\{I_T~|~T\in \mathcal{T}, I_T~{\rm is~a~nonempty~intersection}\}\]
in $\bZ[x_1,\ldots ,x_d]$ is countable.  
Consider the ideals $I_T' = I_T\bQ[x_1,\ldots ,x_d]$.  
By Lemma \ref{lem: amitsur} one of the following holds; either:
\begin{enumerate}
\item there exist $T_1,\ldots ,T_d\in \mathcal{T}$ such that $I_{T_1}'\cap \cdots 
\cap I_{T_d}' \subseteq I \bQ[x_1,\ldots ,x_d]$;  or 
\item there is an ideal $L\supset I\bQ[x_1,\ldots ,x_d]$ of 
$\bQ[x_1,\ldots ,x_d]$ of finite codimension which does not contain 
any of the non-empty $I_T'$. 
\end{enumerate}
In the first case, note that $T_0=T_1\cap \cdots \cap T_d\in \mathcal{T}$ and 
by definition, $I_{T_0}'\subseteq I_{T_1}'\cap \cdots \cap I_{T_d}' \subseteq I\bQ[x_1,\ldots ,x_d] $.
Since $I\bQ[x_1,\ldots ,x_d]\subseteq I_{T_0}$, we see that
 $I\bQ[x_1,\ldots ,x_d]=I_{T_0}$.  It follows that the set of integers $n$ for which $\sigma^n(I)\supseteq J$ is exactly $T_0$. 

If, on the other hand, (2) holds, then
there exists an ideal $L$ of finite codimension which 
contains $I$ but doesn't contain any of the non-empty $I_T$; then  
$L_1=L\cap \mathbb{Z}_p[x_1,\ldots ,x_d]$ is $p$-reduced 
and satisfies
$\bZ[x_1,\ldots ,x_d]/L_1 \otimes_{\bZ} \bQ \cong \bQ[x_1,\ldots ,x_d]/L$ 
is finite-dimensional.  
Hence $I_T \subseteq L_1$ for some $T\in \mathcal{T}$, and so $I_T'\subseteq L$, a contradiction.  
We thus obtain the result. 
\end{proof}

Now we can prove Theorem \ref{thm: polycase}. 

\begin{proof}[Proof of Theorem \ref{thm: polycase}]
Let $K$ be a field of characteristic $0$ and let $A= K[x_1, \ldots , x_d]$, with 
$\sigma: A \to A$ a $K$-algebra automorphism, and let $I$ and $J $ 
be two proper  ideals of $A$.  Then there exist polynomials $F_1,\ldots, F_d, G_1,\ldots ,G_d$ in $A$ such that 
$\sigma(x_i)=F_i(x_1,\ldots ,x_d)$ and
 $\sigma^{-1}(x_i)=G_i(x_1,\ldots, x_d)$ for $i\in\{1,\ldots ,d\}$.  Next there exist polynomials $C_1,\ldots ,C_s$ in $A$ that generate $I$ as an ideal and polynomials $D_1,\ldots ,D_t$ in $A$ that generate $J$.  

We let $\sS$ denote the set of all nonzero elements of $K$ that occur as a coefficient of one of 
$F_1,\ldots ,F_d, G_1,\ldots ,G_d$ and $C_1,\ldots ,C_s, D_1,\ldots ,D_s$ and we let $R_0=\mathbb[ \sS]$ 
denote the finitely generated $\mathbb{Z}$-algebra generated by the elements of $\sS$ inside $K$ and 
we let $K_0$ denote the field of fractions of $R_0$.  By construction, 
$\sigma(R_0[x_1,\ldots ,x_d])\subseteq R_0[x_1,\ldots ,x_d]$ and 
$\sigma^{-1}(R_0[x_1,\ldots ,x_d])\subseteq R_0[x_1,\ldots ,x_d]$, and so
$\sigma$ restricts to an $R_0$-algebra automorphism of $R_0[x_1,\ldots ,x_d]$ and to a $K_0$-algebra automorphism of $K_0[x_1,\ldots ,x_d]$.   

We let 
$I_0:= I\cap K_0[x_1,\ldots ,x_d])$ and $J_0= J\cap K_0[x_1,\ldots ,x_d])$,
which are both $K_0[x_1, x_2,  \ldots, x_d]$- ideals.
We claim that:
\[
 \sigma^n(I)\supseteq J \quad  \mbox{ if and only if} \quad \sigma^n(I_0)\supseteq J_0.
\]
To see this, first observe that if $\sigma^n(I)\supseteq J$, then we have
$\sigma^n(I_0)\supseteq J_0$,
since $\sigma^n(K_0[x_1,\ldots ,x_d])=K_0[x_1,\ldots ,x_d]$.  
Next suppose that $n$ is an integer
 for which $\sigma^n(I_0)\supseteq J_0$. 
Note that $\sigma^n(I)$ is generated by $\sigma^n(C_1),\ldots ,\sigma^n(C_s)$ and each of these generators is in $K_0[x_1,\ldots ,x_d]$.  Since $K[x_1,\ldots ,x_d]$ is a free $K_0[x_1,\ldots ,x_d]$-module, we have that $\sigma^n(I_0)$ is generated by $\sigma^n(C_1),\ldots ,\sigma^n(C_s)$ as an ideal in
$K_0[x_1, x_2, \ldots , x_d]$.   By construction of $K_0$,  the generators $D_i$ of $J$ lie in
$J_0$ for $i\in \{1,\ldots ,t\}$ and so by assumption 
\[
D_i\in \sum_{i=1}^t K_0[x_1,\ldots ,x_d]\sigma^n(C_i) \subseteq \sigma^n(I)
\]
 for $i=1,\ldots ,t$.  It follows that $J\subseteq \sigma^n(I)$,
proving the claim.

By Lemma \ref{le31}, there exists a prime $p \ge 5$ such that $R_0$ embeds in $\mathbb{Z}_p$ as a $\mathbb{Z}$-algebra.  We identify $R_0$ with its image in $\mathbb{Z}_p$ and we similarly identify $K_0$ with its image in $\mathbb{Q}_p$.  
We then see that the restriction of $\sigma$ to $K_0[x_1,\ldots ,x_d]$ lifts to a $\mathbb{Q}_p$-algebra automorphism 
$\taux$ of $\QQ_p [x_1, \ldots , x_d]$, by identifying $\QQ_p [x_1, \ldots , x_d]$ with 
$(K_0[x_1,\ldots ,x_d]\otimes_{K_0} \mathbb{Q}_p)$ and then taking 
$\taux$ to be the map 
$\sigma\otimes {\rm id}$.
In addition, the restriction of $\sigma$ to $R_0[x_1,\ldots ,x_d]$ lifts to a $\mathbb{Z}_p$-algebra automorphism 
$\taux_0$ of $\ZZ_p [x_1, \ldots , x_d]\cong (R_0\otimes_{R_0} \mathbb{Z}_p)[x_1,\ldots ,x_d]$.  
Note that $\mathbb{Q}_p[x_1,\ldots ,x_d]$ is a free 
$K_0[x_1,\ldots ,x_d]$-module and hence the ideals $I_0$ and $J_0$ lift to ideals  
$I'$ and $J'$ of $ \QQ_p[x_1,  \ldots, x_d]$ respectively.
Moreover, since any basis for $\mathbb{Q}_p$ over $K_0$ is fixed by $\taux$, by freeness we have
$$
\{ m \in \ZZ: \Es^m(I_0) \supseteq J_0\} \equiv  \{m \in \ZZ: \taux^m(I') \subset
J' \}.
$$
Next,  we let 
$$\widetilde{I}=I'\cap \ZZ_p[x_1,  \ldots, x_d]$$ and
$$\widetilde{J}=J'\cap \ZZ_p[x_1,  \ldots, x_d].$$  
Since $I', J'$ are
$\QQ_p[x_1, \ldots , x_d]$-ideals,
the ideals $\widetilde{I}, \widetilde{J}$ are necessarily $p$-reduced.
Since $\taux$ restricts to the automorphism $\taux_0$ of $\ZZ_p[x_1,  \ldots, x_d],$ 
a similar argument to the one employed earlier in the proof gives 
$$
 \{m \in \ZZ: \taux^m (I') \subset J' \} \equiv  
 \{m \in \ZZ: \taux_0^m(\widetilde{I}) \subset  \widetilde{J} \}.
$$
Theorem~\ref{th41} now applies to show that the set on the right is the union of a finite
number of complete-doubly infinite arithmetic progressions plus a
finite set;  this equals the set of $n$ with $\Es^n(I_0) \supseteq J_0$. 
Finally, using the claim above,   this set is identical to the set of integers $n$ for which $\sigma^n(I)\supseteq J$,
giving the result. 
\end{proof}

%
%
%
%
\section{Proof of  Main Result}\label{sec5}
\setcounter{equation}{0}

In this section we prove Theorem~\ref{thm:main}.
This is established by reduction of the general case
to the polynomial ring case using the following result of
Srinivas, and then  using Theorem \ref{thm: polycase}.\smallskip


\begin{thm}~\label{th61}
(Srinivas) Let $A$ be a finitely generated algebra over
an infinite field $K$. Then there exists a natural
number $n=n(A)$ such that for all $N > n$, if
\[
f: K[x_1, \ldots, x_N] \to A~~\mbox{and}~~ g: K[x_1, \ldots, x_N] \to A
\]
are two surjective  $K$-algebra homomorphisms, then there
is an elementary $K$-algebra automorphism
$\phi:  K[x_1, \ldots, x_N] \to K[x_1, \ldots, x_N]$ such that 
$f= g \circ \phi.$
\end{thm}

\paragraph{\bf Remark.} Here ``elementary  $K$-algebra automorphism'' means
one that is  a finite
composition of automorphisms of the types: 
\begin{enumerate}
\item[(i)] (Translations) $x_i \mapsto x_i + c_i$, with $c_i \in K$;
\item[(ii)] (Linear transformations)
$x_i \mapsto \sum_j c_{i,j}x_j$ where $(c_{i,j}) \in {\rm GL}_{N}(K)$; 
\item[(iii)](Triangular automorphisms) $x_i \mapsto x_i$ for $1 \le i \le N-1$
and $x_N\mapsto x_N+ F(x_1, x_2, \ldots, x_{N-1})$. 
\end{enumerate}
For our application we only need the fact that $\phi$ is a
$K$-algebra automorphism.

\begin{proof} See Srinivas \cite[Theorem 2, p. 126]{Sr91}.~
\end{proof}

We obtain the following result as an immediate corollary. \smallskip
\begin{prop} \label{prop: 82}
Let $A$ be a finitely generated algebra over a field $K$ and let $\phi:A\rightarrow A$ be a $K$-algebra automorphism.  Then there exists a natural number $N$,  
a surjective $K$-algebra homomorphism
 $\muu: K[x_1,\ldots ,x_N]\rightarrow A$, and a $K$-algebra automorphism $\tilde{\phi}$ of 
$K[x_1,\ldots ,x_N]$ such that:
\begin{enumerate}
\item{$\muu\circ \tilde{\phi}^n=\phi^n\circ \muu$ for all integers $n$;}
\item{If $I$ and $J$ are ideals in $A$, then 
$\phi^n(I)\supseteq J$ if and only if 
\[
\tilde{\phi}^n(\muu^{-1}(I))
\supseteq \muu^{-1}(J).
\]
}
\end{enumerate}
\end{prop} 

\noindent \begin{proof} We pick $n(A)$ as in the statement of Theorem \ref{th61}.
Since $A$ is finitely generated, by $f_1, \ldots , f_r$, say,  we can take $N= \max( r, \,n(A)+1)$
and construct a surjective map $\muu: K[x_1, \ldots , x_N] \rightarrow A$ sending $x_i \mapsto f_i$
for $1 \le i \le r$ and, if $r< n(A)+1$,  sending any extra $x_j$ to $0$. 
Now, since  $\muu$ and $\phi\circ \muu$ are surjective, by Theorem \ref{th61} there exists an
automorphism  
$\tilde{\phi}$ 
of $K[x_1,\ldots ,x_N]$ satisfying
$\muu\circ \tilde{\phi}=\phi \circ \muu$.  
By construction $\muu\circ \tilde{\phi}=\phi\circ\muu$. 
Note that by composing on the left by $\phi^{-1}$ and on the right by $\tilde{\phi}^{-1}$, we see
$$ \muu\circ \tilde{\phi}^{-1}=\phi^{-1}\circ\muu.$$

We now prove property (1)  by induction on $|n|$, noting
the base case $|n|=1$ is established (for both $n=1, -1$).
Assume that it is true for all integers $n$ with $0<|n|< m$.  
Then
$$
\phi^m\circ \muu\  = \  \phi\circ (\phi^{m-1}\circ \muu) \ = 
\ \phi\circ  (\muu\circ \tilde{\phi}^{m-1}) =
(\phi\circ \muu)\circ \tilde{\phi}^{m-1} = \muu\circ \tilde{\phi}^m.$$
A similar argument shows that $\phi^{-m}\circ \muu = \muu\circ \tilde{\phi}^{-m}$,
and so (1)  follows by induction.

To verify (2),  suppose that $I$ and $J$ are ideals of $A$.  
Since $\muu$ is surjective, we have
$\muu(\muu^{-1}(L))=L$ for every ideal of $A$.
If $\phi^n(I)\supseteq J$, then 
$\phi^n\circ \muu(\muu^{-1}(I))\supseteq J$.  
Consequently, $\muu\circ \tilde{\phi}^n(\muu^{-1}(I))\supseteq J$.  Thus
$$
\tilde{\phi}^n(\muu^{-1}(I))\supseteq 
\muu^{-1}(\muu\circ \tilde{\phi}^n(\muu^{-1}(I)))\supseteq \muu^{-1}(J).
$$
Also, if $\tilde{\phi}^n(\muu^{-1}(I))
\supseteq \muu^{-1}(J)$, then 
$\muu(\tilde{\phi}^n(\muu^{-1}(I)))\supseteq \muu(\muu^{-1}(J)) = J$.
Since $\muu\circ \tilde{\phi}^n = \phi^n\circ\muu$, we see
${\phi}^n(I)={\phi}^n(\muu(\muu^{-1}(I))\supseteq J$.  
Thus (2) is established. $~~~$ \end{proof}


\begin{proof}[Proof of Theorem~\ref{thm:main}.]
We are given a field $K$ of characteristic $0$ and a finitely generated commutative $K$-algebra $A$ with a $K$-algebra automorphism $\sigma$.  We wish to show that the set of integers $n$ such that $\sigma^n(I)\supseteq J$ is a finite union of complete doubly-infinite arithmetic progressions along with a finite set.

By Proposition \ref{prop: 82} there exists a polynomial algebra $K[x_1,\ldots, x_N]$ with 
an automorphism $\phi$ and ideals $I'$ and $J'$ such that $$\phi^n(I')\supseteq J'  \iff \sigma^n(I)\subseteq J.$$  The result now follows from Theorem \ref{thm: polycase}. \end{proof}


%
%
%
%

\section{Examples}\label{sec6}
\label{examples}
\setcounter{equation}{0}

We now  give  several examples which show that results for general ideals $I, J$ can
change the answers compared to their associated radical ideals $\sqrt{I}, \sqrt{J}$: allowing non-radical ideals can change
the structure of infinite arithmetic progressions, or eliminate them entirely.  We recall that in a commutative ring $R$, the \emph{radical} is the ideal $I$ consisting of all nilpotent elements of $R$.  In particular, $R/I$ is reduced.  \smallskip

\begin{exam}~\label{exam1}
Let $K$ be an algebraically closed field of characteristic $0$ and let $A=K[x,y,z,t,u]$.  
We define a $K$-algebra automorphism $\sigma$ by
$$
x\mapsto  y-(x-y+yz)t+(x-y+yz),\qquad
 y \mapsto x-y+yz, \qquad
z \mapsto -z,$$
$$
t\mapsto -t, \qquad
u\mapsto u.
 $$
Then $\sigma$ is an automorphism. Let
\begin{eqnarray*} 
J &=& (x^2,xy,y^2,y-x,zt-1,x-u), \\
I &=&(x^2,xy,y^2,y-x,zt-1,x(z-1),x-u).
\end{eqnarray*}
Then 
$$
\sigma^n(\sqrt{I})\supseteq \sqrt{J}~~~\mbox{for all}~~~ n \in \ZZ.
$$
However we have
$$
\sigma^n(I )\supseteq J ~~~\mbox{if and only if }~ ~~n \equiv ~0, 3~(\bmod ~4)
$$
Also
$$
\sigma^n(I ) \supseteq \sqrt{J} ~\mbox{never holds}.
$$
\end{exam} 

\noindent \begin{proof}
We note that $\sigma$ is an automorphism, since it is a composition of automorphisms $\sigma_3\circ \sigma_2\circ \sigma_1$, where
$$
\begin{array} {rclcrclcrcl}
\sigma_1(x) &=&x+yt-y, && \sigma_2(x) &=& y,&& \sigma_3(x) &=& x-y+yz\\
\sigma_1(y)&=&y, && \sigma_2(y) &=& x,&&\sigma_3(y) &=& ~y\\
\sigma_1(z)&=&z, &&\sigma_2(z) &=& z, &&\sigma_3(z) &=& -z\\
\sigma_1(t)&=&t ,&& \sigma_2(t) &=& t, &&\sigma_3(t) &=& -t\\
\sigma_1(u)&=&u, && \sigma_2(u) &=& u, &&\sigma_3(u) &=& ~u\\
\end{array}
$$
We note that $\sqrt{J}\supseteq (x,y,zt-1,u)$, as $x^2,y^2\in I$ and $zt-1, x-u\in J$.  We claim that $\sqrt{J}=(x,y,zt-1,u)$.  To see this, observe that $J\subseteq (x,y,zt-1,u)$ and so $A/J$ surjects onto $A/(x,y,zt-1,u)\cong K[z,z^{-1}]$.  Since $K[z,z^{-1}]$ is reduced, we see that $\sqrt{J}$ must have zero image under this homomorphism and so $\sqrt{J}\subseteq (x,y,zt-1,u)$. Thus $\sqrt{J}= (x,y,zt-1,u),$ which is a $\sigma$-invariant ideal.  
Also $J \subset (x, y, zt-1, xz, u) \subset \sqrt{I}$ hence
$\sigma^n(\sqrt{J})\subseteq \sqrt{I}$ for all $n \in \ZZ$.  Let $L=(x^2,xy,y^2,y-x,zt-1)$.  
We note that $L\subseteq J$ is $\sigma$-invariant.  An easy induction shows that 
$$
\sigma^n(x)\equiv (-1)^{{n-1\choose 2}}xz^{n} ~(\bmod ~L)
$$ 
for 
$n\in \mathbb{Z}$, where we interpret $z^{-i}$ as being $t^i~(\bmod ~L)$.  Thus
$\sigma^n(J)\subseteq I$ if and only if 
$$(-1)^{{n-1\choose 2}}xz^{n} -u\in L+A(x(z-1))+A(x-u).$$
This occurs if and only if 
$$
(-1)^{{n-1\choose 2}}x -u \in L+A(x(z-1))+A(x-u).
$$  
If $n\equiv 0,1~(\bmod ~4)$ then this occurs.  
Note that 
$$
L+A(x(z-1))+A(x-u)\subseteq (z-1,t-1,x^2,y,x-u)
$$ 
and 
$$
A/(z-1,t-1,x^2,y,x-u)\cong K[x]/(x^2).
$$
The image of $(-1)^{{n-1\choose 2}}x -u$ under this homomorphism is 
$(-1)^{{n-1\choose 2}}x -x+(x^2),$ which has nonzero image for $n\equiv 2,3~(\bmod ~4)$. 
 Hence we have two arithmetic progressions 
$n\equiv 0~(\bmod~4)$ and $n\equiv 1~(\bmod ~4)$ for which $\sigma^n(J)\subseteq I$.  Equivalently, $\sigma^n(I)\supseteq J$ if and only if $n\equiv ~0,3~(\bmod~ 4)$. 
 We already saw that $\sigma^n(\sqrt{I})\supseteq \sqrt{J}$ for all $n$.  
 Finally, note that we always have
$\sigma^n(\sqrt{J})\not\subseteq I$, since $\sqrt{J}$ is $\sigma$-stable and $x\not \in I$.
$~~~$  \end{proof} 

We next consider finitely generated commutative $K$-algebras $A$ having a nontrivial radical. 
We give two 
 examples of automorphisms of algebras with nonzero radicals for which the radical affects the dynamics non-trivially.  We begin with a simple example that shows that an automorphism of a ring whose action is trivial on the reduced ring can still produce non-trivial dynamics.\smallskip


\begin{exam}~\label{exam2}
 Let $K$ be a field of characteristic $0$ and let $A=K[x,y,z]/(x,y)^3$.  Let $\sigma:A\rightarrow A$ be the automorphism given by $\sigma(x)=y$, $\sigma(y)=x$, $\sigma(z)=z$.  Let $I=(x)$ and $J=(y)$.  Then 
 $$
 \sigma^n(I)\supseteq J ~~~\mbox{if and only if} ~~~n\equiv 1~(\bmod~ 2).
 $$.
\end{exam}

\noindent \begin{proof}
This holds by inspection. The point of this  example 
is that $\sigma$ induces the trivial automorphism on $K[z]\cong A/J(A)$,
where  $J(A)$ denotes the radical of the zero ideal.
$~~~$\end{proof}

We next give an example of an automorphism of a non-reduced ring $A$ with suitable ideals exhibiting 
nontrivial dynamics, consisting of infinite arithmetic progressions, on both $A$ and its reduction $A/ J(A)$,
but the dynamics differ. \smallskip


\begin{exam}~\label{exam3}
 Let $K$ be a field of characteristic $0$,
  and let $A=K[u,v,y,y^{-1},z,z^{-1}]/(u,v)^3$.  We let $\sigma:A\rightarrow A$ be the automorphism given by 
  $\sigma(u)=uy$, $\sigma(v)=vz$, $\sigma(y)=-y$, $\sigma(z)=z$.  Let 
\begin{eqnarray*}
 J &=& (u+v)\\
 I &=& (y-1,z-1,u+v). 
 \end{eqnarray*}
  Then 
 $$
 \sigma^n(I) \supseteq J ~~~\mbox{ if and only if} \,~~~n\equiv 0,3~(\bmod~ 4).
 $$
\end{exam}

\noindent \begin{proof}
We have 
$$\sigma^n(J)=\left( (-1)^{n-1\choose 2} uy^n+vz^n\right),
$$
which is contained in $I$ if and only if ${n-1\choose 2}\equiv 0~(\bmod~ 2)$.  This occurs exactly when $n\equiv 0,1~(\bmod~ 4)$.  
The result follows.
$~~~$ \end{proof} 
\vskip 2mm

\noindent In Example \ref{exam3}  the ring $A$ is non-reduced and 
  the arithmetic progressions  that occur each have ``gaps'' of length $4$.  
However when one studies  the action of the automorphism $\sigma$ on the reduced ring $A/ J(A) \simeq K[y,y^{-1},z, z^{-1}]$ 
we see that it has order $2$, so that its action on the reduced ring for any ideals $I, J$ will  
have orbits which decompose into arithmetic progressions whose gaps have 
length $1$ or $2$ along with a finite set; in Example \ref{exam3} it is all integers $n$ since $J\subseteq J(A)$.  
An interesting question is whether the ``gaps''  between the two cases can be bounded in terms of the size of 
the gaps that occur from the induced action on the reduced ring  and the degree of 
nilpotency of the radical ideal.\smallskip

We conclude with examples  showing  that the hypotheses that the
commutative $K$-algebra $A$ must be finitely generated over $K$, and that $K$ must have
 characteristic zero, are both needed for the truth of  the theorems above.\medskip


\begin{exam} ~\label{exam4} 
Let $K$ be a field and let $S$ be an arbitrary subset of the 
integers.  Then there exists a commutative $K$-algebra 
$A= A_{S}$ 
which is not Noetherian 
and which is infinitely generated over $K$, having the property
that it has an 
automorphism $\sigma$ and ideals $I$ and $J$ such that 
$\sigma^i(I)\supseteq J$ holds if and only if $i\in S$.
\end{exam}

\noindent \begin{proof}
Let $A=K[x_n~:~n\in\mathbb{Z} ]$ be a polynomial ring in infinitely 
many variables and
let $\sigma$ be the two-sided shift automorphism defined by 
$\sigma(x_i)=x_{i+1}$ for $i\in \mathbb{Z}$.  Given a subset $S$ of integers, we
let $P_S=(x_i~|~-i\in S)$.  Then 
$\sigma^i((x_0))\subseteq P_S$ if and only if $-i\in S$ and so
$$\sigma^n(P_S)\supseteq (x_0)$$ precisely when $n\in S$.   $~~~$\end{proof} 
\vskip 2mm

If we impose the extra requirement that $A$ be Noetherian, 
but allow it to be infinitely generated, there are nontrivial
restrictions on the allowable sets $S \subset \ZZ$ giving ideal inclusions.
A result of Farkas~\cite[Theorem 8]{Fa87} (applied with  $G=\ZZ$)
shows that if $A$ is Noetherian and $\phi$ is an endomorphism of $A$,
then 
the set of natural numbers $i$ such that $\phi^i(I)\subseteq J$ 
must either be the entire set of natural numbers, or 
else  must have a
{\em syndetic} complement; that is, there exists a natural number $d$ such that if $m$ is in the complement 
then there exists $j$ with $1\le j\le d$ such that $m+j$ is also in the complement. For $K$ a field of
characteristic $0$, we do not know any example of an
infinitely generated Noetherian commutative $K$-algebra $A$ with ideals such that the
set $S$ is not a finite union of arithmetic progressions, possibly augmented by a finite set. \smallskip

Our final example shows that  the hypothesis  that the ground field $K$ have characteristic zero is   necessary for 
Theorem \ref{thm:main}  to hold.
In 1953 Lech \cite{Le53} gave a counterexample in
positive characteristic $p$, and we observe it 
applies at  the level of ideals. \smallskip


\begin{exam} ~\label{exam5} (Lech) Let $p$ be a prime, let 
$K=\mathbb{F}_p(t)$ for the finite field $\mathbb{F}_p$, and let $A=K[x,y]$.  Define
$ \sigma : A \rightarrow A$ by $\sigma(x)=tx$ and 
$\sigma(y)=(1+t)y$.  Take
\begin{eqnarray*}
 J &=& (x+y-1)\\
 I &=& (x-1, y-1). 
 \end{eqnarray*}
  Then 
 $$
 \sigma^n(I) \supseteq J ~~~\mbox{ if and only if} \,~~~n \in \{ -1, -p, -p^2, \dots \}.
 $$
\end{exam}

\noindent \begin{proof}
 We show the equlvalent assertion
$\sigma^{n}((x-y+1))\subseteq (x-1,y-1)$ if and only if  $n \in \{1,p,p^2,\ldots \}$.
Note that $\sigma^n(x-y+1) = t^n x - (1+t)^n y + 1$, whence we have the ideal inclusion
 $\sigma^n((x-y+1))\subseteq (x-1,y-1)$ if and only if $t^n-(1+t)^n+1=0$.  
This equation holds if and only if $n$ is a power of $p$. \end{proof}
 \vskip 2mm

In the case $K$ has characteristic $p>0$, Derksen \cite{De07} has further shown that if 
$\sigma$ is a 
linearizable endomorphism of $A=K[x_1,\ldots ,x_d]$, 
then for ideals $I$ and $J$ of $A$, the set of $m$ such that 
 $\sigma^m(I)\subseteq J$ can be classified; in particular this set is always 
 a $p$-\emph{automatic set},
 as defined in  Allouche and Shallit \cite{AS03}.

%
%
%
%

\section{Appendix:   Dynamics of Endomorphisms--Geometric versus Algebraic }
\label{appendix:bugs}
\setcounter{equation}{0}

This appendix presents a result addressing
 the difference between the geometric and algebraic formulations of dynamics of
endomorphisms. This result concerns the difference between the 
geometric action of endomorphisms acting as maps on $S^d$  by forward iteration versus
the algebraic action of endomorphisms acting  on ideals in $R= S[x_1, \ldots, x_d]$
 by backward iteration. It shows that the two actions differ in some circumstances.

 To state it, suppose   that  $S$ is an integral domain, take  $R= S[x_1, \ldots, x_d]$
 and let $\sM$ be a subset of $R$.   The  set of 
$S$-points  cut out by $\sM$ in $S^d$ is 
\[
V(\sM) :=V_{S}(\sM) =  \{ \bs =(s_1, \ldots , s_d) \in S^d: \ev_{\bs}(p(x))=0 ~\mbox{for all} ~ p(x) \in \sM \}.
\]
Of course $V(\sM) = V(I(\sM))$ where $I(\sM)$ is the smallest ideal containing $\sM$. 
Note that for endomorphisms $\tau$ and ideals $I$ that $\tau^{-1}(I)$ is an $R$-ideal
while $\tau(I):= \{ \tau(p(x)): p(x) \in I\}$ need not be an $R$-ideal.

\medskip


\begin{prop} ~\label{prop:bugs}
Let $S$ be an integral domain of any characteristic, and consider the polynomial ring
$R= S[x_1, \ldots, x_d]$.  Let $\tau: R \to R$ be a $S$-algebra endomorphism of $R$. 
 Consider the following possible
relations between ideals $I$ and $J$ of $R$, the $S$-algebra endomorphism $\tau$,
and the dynamical evaluation map $f_{\tau}$:
\begin{enumerate}
\item[ \it (a)]
$\tau^{-1}(I) \supseteq J,$  in the ring  $R$.
\item[\it (b)]
$I \supseteq \tau(J),$  in the ring $R$.
\item[\it (c)]
$V(I) \subseteq V(\tau(J)) $  in $S^d$.
\item[\it (d)]
$f_{\tau}(V(I)) \subseteq V(J)$ in $S^d$.
\item[\it (e)]
$V(I) \subseteq (f_{\tau})^{-1}(V(J))$  in $S^d$.
\end{enumerate}
Then: 
\begin{enumerate}
\item
 There holds 
 $(a) \Leftrightarrow (b)$  and $(c) \Leftrightarrow (d) \Leftrightarrow (e)$.
In addition   $(b) \Rightarrow (c)$, however in  general $(c) \not\Rightarrow (b)$. 
\item
If $S=K$ is an algebraically closed field of any characteristic and
$I$ is a radical ideal then $(c) \Leftrightarrow (b)$. In this case
all five relations (a)--(e) are equivalent.
\end{enumerate}
\end{prop}

This proposition   shows that for  endomorphisms  $\tau $ the obstruction to the equivalence of all
five of 
 these properties  is $(c) \not \Rightarrow (b)$, which concerns  {\em non-radical
 ideals $I$.} This difference matters at the level of the generalized SML theorem for
 ideal inclusion  (Theorem \ref{thm:main})  in that it can change the allowed arithmetic progressions for 
 inclusion  relations of ideals $I, I'$  
 having the same radical ideal $\sqrt{I} = \sqrt{I'}$ with a fixed $J$, cf. Example \ref{exam1}.
 
 \noindent \begin{proof}[Proof of Proposition  \ref{prop:bugs}.]
 Let $\tau$ be an $S$-algebra endomorphism of $R= S[x_1, ..., x_d]$.
 
(1). $(a) \Leftrightarrow (b)$.  For any 
$\sM \subseteq R$, set $\tau^{-1}(Y):= \{ r(x) \in R: \tau(r(x)) \in \sM\}$.
Then we have the inclusions 
\[
\tau \circ \tau^{-1}(\sM) \subseteq \sM \subseteq \tau^{-1} \circ \tau (\sM).
\]
Furthermore  if $\tau:= \sigma$ is an automorphism, then equality holds in
both inclusions.
Suppose $\tau^{-1}(I) \supseteq J$. We apply $\tau$ to both sides
to obtain
$$
I \supseteq \tau \circ \tau^{-1}(I) \supseteq \tau(J).
$$
Conversely, given $I \supseteq \tau(J)$, applying $\tau^{-1}$ to both sides gives
$$
\tau^{-1} (I) \supseteq \tau^{-1} \circ \tau(J) \supseteq J.
$$

$(b) \Rightarrow (c)$. We  are given $ I \supset \tau(J)$. 
Now suppose $\bs \in V(I)=V_{S}(I)$ and we are to show $\bs \in V(\tau(J))$. 
The hypothesis asserts $\ev_{\bs}(q(\bx)) =0$ for all $q(\bx) \in I$.  The
conclusion  asserts that
$$ 
\tau(p) (\bs) := \ev_{\bs}(\tau(p)(\bx)) =0 \quad \mbox{for all} \quad p(\bx) \in J.
$$
To verify this,  the inclusion $\tau(J) \subseteq I$ gives
$$
\tau(p)(x) = p( \tau(x_1), \ldots , \tau(x_d)) =: q(\bx), 
$$
for some $q(\bx) \in I$. 
Now
\[
 \ev_{\bs}(\tau(p))(\bx) = p( \ev_{\bs}(\tau(x_1)), \ldots,\ev_{\bs}(\tau(x_d)) ) 
\] 
and by  definition 
$\ev_{\bs}( \tau(x_1),\ldots, \tau(x_d)) = f_{\tau}(\bs).$
We conclude 
$$
\ev_{\bs}(\tau(p)(\bx))= 
\ev_{f_{\tau}(s)}(p(\bx)).
$$

Now $\bs \in V(I)$ gives  $q(\bs)=0$ whence
\begin{eqnarray*}
0= q(\bs) &:= & \ev_{\bs}(q(\bx)) = \ev_{\bs}(p(\tau(x_1), \ldots  , \tau(x_d)) )\\
 & =& p( \ev_{\bs}( \tau(x_1), \ldots , \tau(x_d)) = \ev_{f_{\tau}(s)}(p(\bx))
\end{eqnarray*} 
Since this holds for all $p(\bx) \in J$, we have $\bs \in V(\tau(J))$.\smallskip

$(c) \Leftrightarrow (d)$. Let $\bs \in S^d$.  The key property is
the identity, valid for all $r(\bx)= r(x_1, \ldots, x_d) \in R$, that
\[
\ev_{\bs} (\tau( r(x_1, \ldots , x_d)))= \ev_{\bs} (r(\tau(x_1), \ldots, r(\tau(x_d)) )= \ev_{f_{\tau}(\bs)}(r(x)).
\]
In what follows we  abbreviate $r(\bs) := \ev_{\bs}(r(x_1, \ldots , x_d))$.

Suppose $V(I) \subseteq V(\tau(J)).$ If $\bs \in V(I)$ then $q(\bs)=0$ for all $q(x) \in I$.
We are to show that $f_{\tau}(V(I)) \subseteq V(J)$, which asserts that 
for all $p(x) \in J$ there holds 
$ p(f_{\tau}(\bs)):=  \ev_{f_{\tau}(\bs)} (p(x))=0.$ Here, using the equality above 
$$
 \ev_{f_{\tau}(\bs)} (p(\bx)) = \ev_{\bs} (\tau(p) (\bx)) =: \tau(p (\bs)).
$$
Now $\tau(p(\bs))=0$ because $V(I) \subset V(\tau(J)).$

Conversely, suppose  $f_{\tau} (V(I)) \subseteq V(J)$ in $S^d$. 
We must show $V(I) \subset V(\tau(J))$. Given $\bs \in V(I)$.
then we must show that for each $p(x) \in J$ that
$$
\tau(p)(\bs):= \ev_{\bs}(\tau(p)(\bx))=0.
$$
Using the identity above 
$\ev_{\bs}(\tau(p)(\bx))=  \ev_{f_{\tau}(\bs)} (p(\bx)).
$
But by hypothesis  $f_{\tau}(\bs) \in V(J)$, whence 
$$
p(f_{\tau}(\bs)) := \ev_{f_{\tau}(\bs)} (p(\bx))=0,
$$
as required.\smallskip

$(d) \Leftrightarrow (e)$. For any subset $\sY \subset S^d$, set
$f_{\tau}^{-1}(\sY) = \{ \bs \in S^d:~ f_{\tau}(\bs) \in \sY\}.$ Then we have the inclusions
\[
f_{\tau} \circ f_{\tau}^{-1}(\sY) \subseteq \sY \subseteq f_{\tau}^{-1} \circ f_{\tau} (\sY).
\]
Furthermore if $\tau :=\sigma$ is an automorphism, we have equality in both inclusions,
for in this case $f_{\sigma}$ is a bijection with inverse $f_{\sigma^{-1}}$.
The argument is similar to $(a) \Leftrightarrow (b)$. Suppose $f_{\tau}(V(I)) \subset V(J)$.
Applying $f_{\tau}^{-1}$ yields
$$
V(I) \subseteq f_{\tau}^{-1} \circ f_{\tau}(V(I)) \subseteq f_{\tau}(V(J))
$$
Conversely, suppose $V(I) \subseteq f_{\tau}^{-1} (V(J)) $. Applying $f_{\tau}^{-1}$ yields
$$
f_{\tau}(V(I)) \subseteq f_{\tau} \circ f_{\tau}^{-1} (V(J)) \subseteq V(J)
$$

$(c) \not \Rightarrow (b)$. This is well known.
Take  $R=K[x]$, for $K$ a field. and $\tau(x) = x^3.$ Take $I = (x^4)$ and $J= (x).$
Then $ x^3 \in \tau(J) \subset (x^3)$, so that $V(I) = V(\tau(J)) = \{ 0\},$ and $V(I) \subseteq V(\tau(J)).$
But $x^3 \not\in I$, so $I \not\supseteq \tau(J)$. (Note also that  $\tau(J)=\{ \tau(p)(x): p(x) \in J\}$ is not an $R$-ideal.)\smallskip

(2). The assertion $(c) \Rightarrow (b)$ when $S=K$ is an algebraically closed
field (of any characteristic) and $I$ is a radical ideal  is the Nullstellensatz. 
This gives all the equalities $(a)$--$(e)$.    No condition is imposed on the ideal $J$ to get the equality.

 \end{proof} 
 
 \begin{rem} 
  {\em The  failure of the methods of this paper to handle general endomorphisms arises 
  not from Proposition \ref{prop:bugs}, but rather from 
 the failure of the generalized $p$-adic analytic arc theorem to apply to  certain endomorphisms.}
\end{rem} 

\end{document}